\documentclass[11pt]{article}

\usepackage[utf8]{inputenc}

\usepackage{geometry}
\usepackage{graphicx}

\usepackage{booktabs} 
\usepackage{array} 
\usepackage{paralist} 
\usepackage{verbatim} 
\usepackage{subfig} 

\usepackage{fancyhdr} 
\usepackage[nottoc,notlof,notlot]{tocbibind}
\usepackage[titles,subfigure]{tocloft}

\usepackage{amsmath,amsfonts,amsthm,mathrsfs,amssymb,cite}
\usepackage[usenames]{color}

\newtheorem{thm}{Theorem}[section]

\newtheorem{lem}{Lemma}[section]

\theoremstyle{definition}

\theoremstyle{remark}

\newtheorem{rem}{Remark}[section]
\numberwithin{equation}{section}

\title{\bf Two Single-shot Methods for Locating Multiple Electromagnetic Scatterers}
\author{Jingzhi Li\thanks{Faculty of Science, South University of Science and
Technology of China, Shenzhen 518055, P.~R.~China. Email: {\tt li.jz@sustc.edu.cn}},\quad Hongyu Liu\thanks{Department of Mathematics and Statistics, University of North Carolina, Charlotte, NC 28223, USA.   Email:  {\tt hongyu.liuip@gmail.com}},\quad Zaijiu Shang\thanks {Institute of Mathematics, Academy of Mathematics and
Systems Science, Chinese Academy of Sciences, Beijing 100190, P. R. China. Email: {\tt zaijiu@amss.ac.cn} } ,\quad Hongpeng Sun\thanks{Institute for Mathematics and Scientific Computing, University of Graz, Heinrichstr. 36, A-8010 Graz, Austria. Email: {\tt hongpeng.sun@uni-graz.at} } }

\date{} 

\begin{document}
\maketitle

\begin{abstract}

We develop two inverse scattering schemes for locating multiple
electromagnetic (EM) scatterers by the electric far-field
measurement corresponding to a single incident/detecting plane wave.
The first scheme is for locating scatterers of small size compared
to the wavelength of the detecting plane wave. The multiple
scatterers could be extremely general with an unknown number of
components, and each scatterer component could be either an
impenetrable perfectly conducting obstacle or a penetrable
inhomogeneous medium with an unknown content. The second scheme is
for locating multiple perfectly conducting obstacles of regular size
compared to the detecting EM wavelength. The number of the obstacle
components is not required to be known in advance, but the shape of
each component must be from a certain known admissible class. The
admissible class may consist of multiple different reference
obstacles. The second scheme could also be extended to include the
medium components if a certain generic condition is satisfied. Both
schemes are based on some novel indicator functions whose indicating
behaviors could be used to locate the scatterers. No inversion will
be involved in calculating the indicator functions, and the proposed
methods are every efficient and robust to noise. Rigorous
mathematical justifications are provided and extensive numerical
experiments are conducted to illustrate the effectiveness of the
imaging schemes.

\end{abstract}

\section{Introduction}\label{sec:intro}

This paper is concerned with the inverse scattering problem of reconstructing an inhomogeneous scatterer located in an otherwise homogeneous space by measuring the corresponding electromagnetic (EM) wave fields far away from the scatterer produced by sending some detecting EM wave fields. The inverse electromagnetic scattering problem has been playing a key role in many areas of science and technology, such as radar and sonar, non-destructive testing, remote sensing, geophysical exploration and medical imaging to name just a few; see \cite{AK1,AK2,CaC05,CCM11, CK,Isa,LZr,Pot01,Uhl} and the references therein. In the current article, we shall mainly consider the reconstruction scheme for this inverse scattering problem. There are extensive studies in the literature in this aspect and many imaging schemes have been developed by various authors, and we would like to refer to \cite{ABP06,Ammari2,AILP,AK1,AK2,AKKL,APRT,CZ,CoK96,Ike99,IJZ0,IJZ,KG,Pot06,SZC} and the references therein. However, we would like to remark that in those schemes, one either needs to make use of many wave measurements or if only a few wave measurements are utilized, then one must require that the underlying scatterer is of small size compared to the detecting wavelength.  In this work, we shall consider our study in a very practical setting by making use of a single electric far-field measurement. That is, we shall consider the reconstruction by measuring the far-field electric wave corresponding to a single time-harmonic plane wave.  From a practical viewpoint, the inverse scattering method with a single far-field measurement would be of significant interests, but is highly challenging with very limited progress in the literature; we refer to \cite{CK,Isa,L,LZr,LYZ,Uhl} for related discussion and surveys on some existing development. For more practical considerations, we shall work in a even more challenging setting by assuming very little {\it a priori} knowledge of the underlying scattering object, which might consist of multiple components, and both the number of the components and the physical property of each component are unknown in advance. This setting would make our current study even more highly non-trivial, but on the other hand would be of significant practical importance when there is little {\it a priori} information available on the target object. Before we proceed to discuss more about our results, we next present the mathematical framework that we shall work within.

Let $\Omega$ be a bounded $C^2$ domain in $\mathbb{R}^3$ which supports an inhomogeneous isotropic EM medium characterized by the electric permittivity $\varepsilon(x)$, magnetic permeability $\mu(x)$, and conductivity $\sigma(x)$. Both $\varepsilon(x)$ and $\mu(x)$ are positive scalar functions and $\sigma(x)$ is a non-negative scalar function. It is assumed that $\Omega_e:=\mathbb{R}^3\backslash\overline{\Omega}$ is connected and $\Omega_e$ represents the uniformly homogeneous background space. Without loss of generality, we assume that $\Omega_e$ is the vacuum, i.e., $\varepsilon(x)=\mu(x)=1$ and $\sigma(x)=0$ for $x\in\Omega_e$. But we would like to emphasize that all our subsequent studies could be straightforwardly extended to the case that $\Omega_e$ is not necessarily the vacuum but uniformly homogeneous. The inverse problem that we shall consider in this paper is to recover the inhomogeneity $(\Omega; \varepsilon,\mu,\sigma)$ by measuring the wave field far away from the inhomogeneity produced by sending a single detecting EM wave. In the physical situation, the inhomogeneity $(\Omega; \varepsilon,\mu,\sigma)$ is referred to as a {\it scatterer}, and throughout we shall take the detecting/incident EM wave to be the time-harmonic plane wave,
\begin{equation}\label{eq:plane waves}
E^i(x)=pe^{i\omega x\cdot\theta'},\quad H^i(x)=\frac{1}{i\omega} \nabla\wedge E^i(x),\quad x\in\mathbb{R}^3\,,
\end{equation}
where $\omega\in\mathbb{R}_+$ denotes the frequency, $\wedge$ the exterior product, $\theta'\in\mathbb{S}^{2}$  the impinging direction, and $p\in\mathbb{R}^3$  the polarization with $p\cdot\theta'=0$. The EM wave scattering is governed by the Maxwell equations,
\begin{equation}\label{eq:Maxwell}
\begin{cases}
\displaystyle{\nabla\wedge E-i\omega\mu H=0}\ &\hspace*{-2cm} \mbox{in\ \ $\mathbb{R}^3$},\\
\displaystyle{\nabla\wedge H+i\omega\left(\varepsilon+i\frac{\sigma}{\omega} \right)E=0}\ &\hspace*{-2cm} \mbox{in\ \ $\mathbb{R}^3$},\\
E^-=E|_{\Omega},\quad E^+=(E-E^i)|_{\Omega_e},\\
\displaystyle{\lim_{|x|\rightarrow+\infty}|x|\left| (\nabla\wedge E^+)(x)\wedge\frac{x}{|x|}- i\omega E^+(x) \right|=0},
\end{cases}
\end{equation}
where $E$ and $H$ are the respective electric and magnetic fields, $E^-$ is the interior electric field,  $E^+$ is known as the scattered electric field (cf. \cite{CK,Ned}), and $H^-$ and $H^+$ are their interior and scattered magnetic counterparts.
Moreover, $E^+$ admits the following asymptotic development with the big $\mathcal{O}$ notaion
\begin{equation}\label{eq:farfield}
E^+(x)=\frac{e^{i\omega|x|}}{|x|} A\left(\frac{x}{|x|};\theta',p,\omega\right)+\mathcal{O}\left(\frac{1}{|x|^2}\right)
\end{equation}
uniformly in all directions ${x}/{|x|}$ as $|x|\rightarrow+\infty$. $A(\theta;\theta',p,\omega)$ with $\theta:=x/|x|\in\mathbb{S}^2$ is known as the electric far-field pattern. Here, we would like to distinguish two cases with $0\leq\sigma<+\infty$ and $\sigma=+\infty$. Following the terminologies in the literature, in the former case, $(\Omega;\varepsilon,\mu,\sigma)$ is called an EM medium inclusion and the corresponding inverse problem is referred to as {\it inverse medium scattering}, whereas in the latter case as $\sigma\rightarrow+\infty$, the medium becomes perfectly conducting, and $(\Omega;\varepsilon,\mu,+\infty)$ is called a perfect conductor or a perfectly electric conducting (PEC) obstacle, and the corresponding inverse problem is usually referred to as {\it inverse obstacle scattering}. For a perfectly conducting obstacle, both interior fields $E^-$ and  $H^-:=H|_{\Omega}$ would vanish inside $\Omega$, i.e., the EM fields cannot penetrate inside the object and are governed by
\begin{equation}\label{eq:Maxwell}
\begin{cases}
\displaystyle{\nabla\wedge E-i\omega H=0}\ &\hspace*{-2cm} \mbox{in\ \ $\Omega_e$},\\
\displaystyle{\nabla\wedge H+i\omega E=0}\ &\hspace*{-2cm} \mbox{in\ \ $\Omega_e$},\\
\displaystyle{\nu\wedge E=0}\ &\hspace*{-2cm} \mbox{on\ $\partial\Omega$},\\
\displaystyle{E^+=E-E^i}\ &\hspace*{-2cm} \mbox{in\ \ $\Omega_e$},\\
\displaystyle{\lim_{|x|\rightarrow+\infty}|x|\left| (\nabla\wedge E^+)(x)\wedge\frac{x}{|x|}- i\omega E^+(x) \right|=0},
\end{cases}
\end{equation}
where $\nu$ denotes the outward unit normal vector to $\partial\Omega$. The scattered field $E^+$ has the same asymptotic development as that in \eqref{eq:farfield}. Throughout the rest of the paper, we shall unify the notation of denoting a scatterer by $(\Omega; \varepsilon,\mu,\sigma)$, which could be an inclusion by taking $0\leq \sigma<+\infty$, or an obstacle by just taking $\sigma=+\infty$ disregarding the parameters $\varepsilon$ and $\mu$.

The inverse scattering problem that we consider is to recover $(\Omega;\varepsilon,\mu,\sigma)$ from the knowledge of $A(\theta;\theta',p,\omega)$ or to recover $\Omega$ from the knowledge of $A(\theta;\theta',p,\omega)$ if $\sigma=+\infty$. We will make use of single far-field measurement, i.e., $A(\theta;\theta',p,\omega)$ for all $\theta\in\mathbb{S}^2$ but fixed $\theta',p$ and $\omega$. Furthermore, we shall require very little {\it a priori} knowledge of the underlying scatterer, which could be composed of multiple components, and the number of the components is not required to be known in advance, and each component could be either a medium inclusion or an obstacle. Generically, in this extremely general setting, one cannot expect to recover all the details of the underlying scatterer by using only a single measurement. Instead, we would consider the recovery of locating the multiple components in the setting described above. Nevertheless, at this point, we would like to remark that our numerical experiments indicate that our proposed imaging scheme could also qualitatively reveal the supports/shapes of the scatterer components.

Specifically, two single-measurement locating schemes would be proposed for two separate cases depending on the size of the target scatterer. The first scheme is for locating scatterers of small size compared to the wavelength of the detecting plane wave. The multiple scatterers could be extremely general with very little {\it a priori} knowledge required. Each scatterer component could be either an impenetrable perfectly conducting obstacle or a penetrable inhomogeneous medium with an unknown content, and the number of the scatterer components is not required to be known in advance. The locating scheme is based on a novel indicator function $I_s(z)$ for $z\in\mathbb{R}^3$. If $z$ happens to be the location point of a scatterer component, then $z$ is a local maximum point for $I_s(z)$. Using the indicating behavior of $I_s(z)$, one could then locate all the scatterer components. In defining the indicator function, only inner-products of the electric far-field pattern and the vectorial spherical harmonics are involved. The indicating behavior is derived from the asymptotic expansion of the electric far-field pattern for small `point-like' scatterers. The expansion is further based on the low frequency EM scattering asymptotics (i.e., Raleigh approximation); see \cite{AN,AVV,APRT,DK,Mar}. But for our imaging scheme, the expansion would be formulated in terms of the vectorial spherical harmonics instead of the polarizability tensors. It is also interesting to mention that our numerical experiments show that the proposed scheme also works effectively for locating `partially-small' scatterer, namely, the scatterer is not `point-like' but `line-segment-like'. Furthermore, in addition to finding the locations of the scatterer components, the proposed scheme also shows some very promising feature in qualitatively imaging the supports/shapes of the unknown scatterers.

Our second scheme is for locating multiple perfectly conducting obstacles whose sizes are comparable to the detecting EM wavelength. For this case, we would require that the shape of each component must be from a certain known admissible class. That is, there might be multiple obstacles with an unknown number of components, but the shape of each obstacle component must be from a certain class of reference obstacles which is known in advance. Nevertheless, there could be multiple different reference obstacles. Other than the assumptions made above, no further {\it a priori} information is required for the multiple unknown scatterers. The number of the unknown scatterer components could be larger than the number of the reference obstacles; that is, some of the unknown scatterers possess the same shape. Moreover, it is not necessary for all the reference obstacles to be presented in the unknown scatterer. The setting considered would be of significant practical interests, e.g., in radar and sonar imaging. The second imaging scheme is based on $l'$ indicator functions, where $l'\in\mathbb{N}$ denotes the number of the reference obstacles. Similar to the first imaging scheme, in calculating the indicator functions, only inner products are involved and no inversion will be required. The proposed method is very efficient and robust to measurement noise, and could also be extended to include inhomogeneous medium components if certain generic condition is satisfied. To our best knowledge, this is the first imaging scheme in the literature for locating multiple unknown scatterers of regular size by a single EM far-field measurement. Rigorous mathematical justifications are provided for both schemes. In our theoretical analysis, we would impose the sparse distribution of the scatter components. However, we have conducted extensive numerical experiments, and numerical results show that even without the sparsity assumption, our proposed locating schemes still work very effectively.

The rest of the paper is organized as follows. In Section 2, we
develop the locating schemes, respectively, for the two separate
cases with small scatterers and regular-sized scatterers. Section~3
is devoted to the proofs of Theorems~\ref{thm:main1} and
\ref{thm:main2} stated in Section 2 to justify the indicating
behaviors of the indicator functions for the two locating schemes
proposed in Section 2. In Section 4, we present extensive numerical
experiments to illustrate the effectiveness of the proposed methods.
The work is concluded in Section~\ref{sec:Conclusion} with some
further discussions.

\section{The locating schemes}\label{sect:2}

In this section, we shall develop the two schemes for locating, respectively, multiple scatterers of small size and regular size. In order to ease the exposition, throughout the rest of the paper, we assume that $\omega\sim 1$, and hence the size of a scatterer can be expressed in terms of its Euclidean diameter.

\subsection{Locating small scatterers}\label{sect:small}

We first introduce the class of small EM scatterers for our current study. Let $l\in\mathbb{N}$ and let $D_j$, $1\leq j\leq l$, be bounded $C^2$ domains in $\mathbb{R}^3$. It is assumed that all $D_j$'s are convex and contain the origin. For $\rho\in\mathbb{R}_+$, we let $\rho D_j:=\{\rho x| x\in D_j\}$ and set
\[
\Omega_j=z_j+\rho D_j,\quad z_j\in\mathbb{R}^3,\ \ 1\leq j\leq l.
\]
Each $\Omega_j$ is referred to as a scatterer component and its content is given by $\varepsilon_j,\mu_j$ and $\sigma_j$. It is assumed that $\varepsilon_j>0, \mu_j>0$ and $\sigma_j\geq 0$ are all constants, except the case that $\sigma_j=+\infty$.  As remarked in Section~\ref{sec:intro}, if $\sigma_j$ is taken to be $+\infty$, then $D_j$ would be regarded as a perfectly conducting obstacle disregarding the parameters $\varepsilon_j$ and $\mu_j$. In the sequel, we shall reserve the letter $l$ to denote the number of components of a scatterer defined by
\begin{equation}\label{eq:scatterer}
\Omega=\bigcup_{j=1}^l \Omega_j\quad \mbox{and}\quad (\Omega;\varepsilon,\mu,\sigma)=\bigcup_{j=1}^l (\Omega_j; \varepsilon_j,\mu_j,\sigma_j).
\end{equation}
The parameter $\rho\in\mathbb{R}_+$ represents the relative size of the scatterer (or, more precisely, each of its component). We now make the following qualitative assumptions,
\begin{equation}\label{eq:qualitative assumptions}
\rho\ll 1\qquad \mbox{and}\qquad \mbox{dist}(z_j, z_{j'})\gg 1\quad \mbox{for\ $j\neq j'$, $1\leq j, j'\leq l$}.
\end{equation}
The assumption \eqref{eq:qualitative assumptions} means that compared to the wavelength of the detecting/incident wave, the relative size of each scatterer component is small and if there are multiple components, they must be sparsely distributed. Our numerical experiments in Section 4 could be more quantitative in this aspect, and it is numerically shown that if the size of the scatterer component is smaller than half a wavelength and the distance between two distinct components is bigger than half a wavelength, one could have a fine reconstruction by using our proposed scheme. In this sense, the qualitative assumption~\eqref{eq:qualitative assumptions} is needed only for our subsequent theoretical analysis of the proposed locating scheme. Furthermore, we would like to emphasize that most of the other restrictive assumptions introduced above are also mainly for the purpose of the theoretical justification, and the numerical experiments in Section 4 will show that our proposed imaging scheme works in a much more general setting. Specifically, the regularity of each reference component $D_j$ is not necessarily $C^2$-smooth, and it could be a Lipschitz domain, and moreover, it is not necessarily convex and not even simply connected, as long as it contains the origin and $\rho D_j$ for a small $\rho\in\mathbb{R}_+$ yields an appropriate domain of small size. Moreover, the content of each medium component is not necessarily constant, and it could be variable. Some of those restrictive assumptions could be relaxed from our theoretical justification in the following. However, it is our focus and emphasis on developing the locating scheme in the present work, and we would not appeal for a most general theoretical study.

In the sequel, without loss of generality, we can assume that there exists $l'\in \mathbb{N}\cup \{0\}$, $0\le l' \le l$ such that
\[
0\leq \sigma_j<+\infty\ \ \mbox{for\ \ $1\leq j\leq l'$}\qquad\mbox{and}\qquad \sigma_j=+\infty\ \ \mbox{for\ \ $l'+1\leq j\leq l$}.
\]
Then, we set
\begin{equation}\label{eq:scatterer2}
\Omega_m:=\bigcup_{j=1}^{l'}\Omega_j\qquad\mbox{and}\qquad \Omega_o:=\bigcup_{j=l'+1}^l \Omega_j
\end{equation}
to denote, respectively, the medium component and the obstacle component of a scatterer. We emphasize that both $l$ and $l'$ are unknown in advance, and $l'$ could be $0$ or $l$, corresponding to the case that $\Omega_m=\emptyset$ or $\Omega_o=\emptyset$. Moreover, the contents and the shapes of the two components are unknown in advance either and they are not even necessarily identical to each other, i.e.,
\[
(\Omega_j;\varepsilon_j,\mu_j,\sigma_j) \neq (\Omega_{j'}; \varepsilon_{j'},\mu_{j'}, \sigma_{j'}) \quad \mbox{for\ \ $j\neq j'$}.
\]
Corresponding to a single EM detecting plane wave \eqref{eq:plane waves}, the electromagnetic scattering is governed by the following Maxwell system
\begin{equation}\label{eq:Maxwell general}
\begin{cases}
\displaystyle{\nabla\wedge E-i\omega\left(1+\sum_{j=1}^{l'}(\mu_j-1)\chi_{\Omega_j} \right) H=0}\ &  \mbox{in\ \ $\mathbb{R}^3\backslash\overline{\Omega}_o$},\\
\displaystyle{\nabla\wedge H+\left(i\omega(1+\sum_{j=1}^{l'}(\varepsilon_j-1)\chi_{\Omega_j})- \sum_{j=1}^{l'}\sigma_{j}\chi_{\Omega_j} \right)E=0}\ & \mbox{in\ \ $\mathbb{R}^3\backslash\overline{\Omega}_o$},\\
E^-=E|_{\Omega_m},\ E^+=(E-E^i)|_{\mathbb{R}^3\backslash\overline{\Omega}},\ H^+=(H-H^i)|_{\mathbb{R}^3\backslash\overline{\Omega}}\,,\\
\nu\wedge E^+=-\nu\wedge E^i\hspace*{1cm}\mbox{on \ \ $\partial \Omega_o$},\\
\displaystyle{\lim_{|x|\rightarrow+\infty}|x|\left| (\nabla\wedge E^+)(x)\wedge\frac{x}{|x|}- i\omega E^+(x) \right|=0}.
\end{cases}
\end{equation}
It is known that there exists a unique pair of solutions $(E^+, H^+)\in H_{loc}(\text{curl}; \mathbb{R}^3\backslash\overline{\Omega}_0)\wedge H_{loc}(\text{curl}; \mathbb{R}^3\backslash\overline{\Omega}_0)$ to the Maxwell equations \eqref{eq:Maxwell general} (see \cite{Ned}). In the following, we shall write
\[
A(\theta;\Omega):=A(\theta; \bigcup_{j=1}^l (\Omega_j; \varepsilon_j, \mu_j, \sigma_j), \theta', p, \omega)
\]
to denote the electric far-field pattern corresponding to the EM wave fields in \eqref{eq:Maxwell general}. We shall also write for $1\leq j\leq l$,
\[
A(\theta;\Omega_j):=A(\theta; (\Omega_j; \varepsilon_j, \mu_j, \sigma_j), \theta', p, \omega)
\]
to denote the far-field pattern corresponding solely to the scatterer $(\Omega_j; \varepsilon_j, \mu_j, \sigma_j)$. Both $A(\theta; \Omega)$ and $A(\theta; \Omega_j)$ are real analytic functions on the unit sphere $\mathbb{S}^2$ (cf. \cite{CK,Ned}).

Next, we introduce the space of $L^2$ tangential fields on the unit sphere as follows,
\[
T^2(\mathbb{S}^2):=\{\mathbf{a}\in\mathbb{C}^3|\ \mathbf{a}\in L^2(\mathbb{S}^2)^3, \ \theta\cdot \mathbf{a}=0\ \ \mbox{a.e. on $\mathbb{S}^2$}\}.
\]
Note that $T^2(\mathbb{S}^2)$ is a linear subspace of $L^2(\mathbb{S}^2)^3$,  the space of vector $L^2$-fields on the unit sphere $\mathbb{S}^2$.
In addition, we recall the vectorial spherical harmonics
\begin{equation}\label{eq:orthonormal system}
\left\{
\begin{split}
U_n^m(\theta)&:=\frac{1}{\sqrt{n(n+1)}}\text{Grad}\, Y_n^m(\theta) \\
V_n^m(\theta)&:=\theta\wedge U_n^m(\theta)
\end{split} \right.
 \qquad  n\in\mathbb{N},\  \ m=-n,\cdots,n,
\end{equation}
which form a complete orthonormal system in $T^2(\mathbb{S}^2)$. In \eqref{eq:orthonormal system}, $Y_n^m(\theta)$ with $\theta\in\mathbb{S}^2$, $m=-n,\ldots,n$ are the spherical harmonics of order $n\geq 0$, and $\text{Grad}$ denotes the surface gradient operator on $\mathbb{S}$. It is known that $A(\theta; \Omega)$ and $A(\theta;\Omega_j)$, $1\leq j\leq l$ all belong to $T^2(\mathbb{S}^2)$. We define
\begin{equation}\label{eq:indicator value}
K^j:=\frac{\|A(\theta;\Omega_j)\|^2_{T^2(\mathbb{S}^2)}}{\|A(\theta;\Omega)\|^2_{T^2(\mathbb{S}^2)}},\quad 1\leq j\leq l,
\end{equation}
and
\begin{equation}\label{eq:indicator function}
\begin{split}
I_s(z):=\frac{1}{\|A(\theta;\Omega)\|^2_{T^2(\mathbb{S}^2)}}\sum_{m=-1,0,1}\bigg( & {\bigg|\left\langle A(\theta;\Omega), e^{i\omega (\theta'-\theta)\cdot z}\, U_1^m(\theta)  \right\rangle_{T^2(\mathbb{S}^2)}\bigg|^2}\\
& +{\bigg|\left\langle A(\theta;\Omega), e^{i\omega(\theta'-\theta)\cdot z}\, V_1^m(\theta)  \right\rangle_{T^2(\mathbb{S}^2)}\bigg|^2}         \bigg),
\end{split}
\end{equation}
where $z\in\mathbb{R}^3$, and $\langle \mathbf{u}, \mathbf{v}\rangle_{T^2(\mathbb{S}^2)}=\int_{\mathbb{S}^2} \mathbf{u}\cdot\overline{\mathbf{v}}\ ds_{\theta}$. Clearly, $K^j$ is a real number, whereas $I_s(z)$ is a real-valued function depending on the point $z\in\mathbb{R}^3$. We are now ready to present
the first main result of this paper, whose proof will be postponed to the next section.
\begin{thm}\label{thm:main1}
Let $(\Omega;\varepsilon,\mu,\sigma)$ be given by \eqref{eq:scatterer} satisfying \eqref{eq:qualitative assumptions}, and $K^j$, $1\leq j\leq l$, and $I_s(z)$ be defined in \eqref{eq:indicator value} and \eqref{eq:indicator function}, respectively. Set
\[
L=\min_{1\leq j, j'\leq l, j\neq j'}\text{\emph{dist}}(z_j, z_{j'})\gg 1.
\]
Then
\begin{equation}\label{eq:independence}
K^j=K_0^j+\mathcal{O}\left( \frac{1}{L}+\rho \right ),\quad 1\leq j\leq l,
\end{equation}
where $K_0^j$ is a positive number independent of $L$ and $\rho$.
Moreover there exists an open neighborhood of $z_j$, $\text{{neigh}}(z_j)$, $1\leq j\leq l$, such that
\begin{equation}\label{eq:thm11}
I_s(z)\leq K_0^j+\mathcal{O}\left(\frac 1 L+\rho \right )\quad\mbox{for\ \ $z\in \text{neigh}(z_j)$},
\end{equation}
and moreover $I_s(z)$ achieves its maximum at $z$ in $\text{neigh}(z_j)$,
\begin{equation}\label{eq:thm12}
I_s(z_j)=K_0^j+\mathcal{O}\left(\frac 1 L+\rho \right ),
\end{equation}
\end{thm}

\begin{rem}\label{rem:indicator}
Clearly, $I_s(z)$ possesses the indicating behavior which could be used to identify the location point $z_j$'s of the scatterer components $\Omega_j$'s. Such behavior is more evident if one considers the case that $\Omega$ has only one component, i.e., $l=1$. In the one-component case, one would have that
\[
I_s(z)<1+\mathcal{O}(\rho)\quad\mbox{for all \ $z\in\mathbb{R}^3\backslash\{z_1\}$},
\]
but
\[
I_s(z_1)=1+\mathcal{O}(\rho).
\]
That is, $z_1$ is a global maximum point for $I(z)$.
\end{rem}

Using Theorem~\ref{thm:main1}, we can formulate our first imaging scheme of locating multiple small scatterer components as follows, which shall be referred to as {\it Scheme S} in the rest of the paper.

\medskip

\noindent{\bf The first single-shot locating method: \textbf{Scheme S}}

\begin{enumerate}

\item[{\bf Step 1.}] For an unknown EM scatterer $\Omega$ in \eqref{eq:scatterer}, collect the far-field data by sending a
 single detecting EM plane wave specified by \eqref{eq:plane waves}.

\item[{\bf Step 2.}] Select a sampling region with a mesh $\mathcal{T}_h$ containing $\Omega$.

\item[{\bf Step 3.}] For each point $z\in \mathcal{T}_h$, calculate $I_s(z). $

\item[{\bf Step 4.}] Locate all the significant local maxima of $I_s(z)$ on $\mathcal{T}_h$, which represent the locations of the scatterer components.
\end{enumerate}

\subsection{Locating regular-sized scatterers}

In this section, we consider the locating of multiple scatterers of regular size. We shall present our scheme for the case that all the scatterer components are perfectly conducting obstacles. Later, we would remark how our method could be extended to include the inhomogeneous medium components.
Let $M_j\subset\mathbb{R}^3$, $1\leq j\leq l$, be bounded simply connected $C^2$ domains that contain the origin.
Let
\begin{equation}\label{eq:regular scatterer}
Q:=\bigcup_{j=1}^l Q_j=\bigcup_{j=1}^l z_j+M_j,\quad z_j\in\mathbb{R}^3,\ \ 1\leq j\leq l.
\end{equation}
Each $Q_j:=z_j+M_j$ is a scatterer component, and it is assumed to be a perfectly conducting obstacle. Illuminated by a single incident EM plane wave \eqref{eq:plane waves}, the EM scattering by the scatterer $Q$ can be described by the following Maxwell system
\begin{equation}\label{eq:Maxwell general regular}
\begin{cases}
\displaystyle{\nabla\wedge E-i\omega H=0}\ & \hspace*{-2cm} \mbox{in\ \ $\mathbb{R}^3\backslash\overline{Q}$},\\
\displaystyle{\nabla\wedge H+i\omega E=0}\ & \hspace*{-2cm} \mbox{in\ \ $\mathbb{R}^3\backslash\overline{Q}$},\\
\ E^+=(E-E^i)|_{\mathbb{R}^3\backslash\overline{Q}},\\
\nu\wedge E^+=-\nu\wedge E^i & \hspace*{-2cm} \mbox{on \ \ $\partial Q$},\\
\displaystyle{\lim_{|x|\rightarrow+\infty}|x|\left| (\nabla\wedge E^+)(x)\wedge\frac{x}{|x|}- i\omega E^+(x) \right|=0}.
\end{cases}
\end{equation}
In the sequel, we write
\begin{equation}\label{eq:farfield regular}
A(\theta;Q):=A(\theta; \bigcup_{j=1}^l Q_j)=A(\theta;\bigcup_{j=1}^l z_j+M_j)
\end{equation}
to denote the far-field pattern corresponding to the EM fields in \eqref{eq:Maxwell general regular}. It is assumed that
\begin{equation}\label{eq:regular assumption}
\mbox{diam}(Q_j)=\mbox{diam}(M_j)\sim 1, \ \ 1\leq j\leq l;
\end{equation}
and
\begin{equation}\label{eq:spare regular}
L=\min_{1\leq j, j'\leq l, j\neq j'}\text{dist}(z_j, z_{j'})\gg 1.
\end{equation}
That is, the size of the underlying scatterer components are comparable to the wavelength of the detecting EM plane wave. This is in sharp difference from our study in Section~\ref{sect:small}, where the scatterer components are of small size compared to the detecting EM wavelength. The qualitative condition \eqref{eq:spare regular} states that the scatterer components are sparsely distributed, and we would like to emphasize again that this is mainly needed for our subsequent theoretical analysis of the proposed locating scheme. Our numerical examples in Section 4 shows that as long as the distance between different components are bigger than half a wavelength, the proposed scheme would yield a fine reconstruction.
Furthermore, we introduce an admissible reference scatterer space
\begin{equation}\label{eq:reference space}
\mathscr{S}:=\{\Sigma_j\}_{j=1}^{l'},
\end{equation}
where each $\Sigma_j\subset\mathbb{R}^3$ is a bounded simply connected $C^2$ domain that contains the origin and
\begin{equation}\label{eq:assumption 1}
\Sigma_j\neq \Sigma_{j'},\quad \mbox{for}\ \ j\neq j',\  1\leq j, j'\leq l'.
\end{equation}
For the present study, we assume that
\begin{equation}\label{eq:shape known}
M_j\in\mathscr{S},\quad j=1,2,\ldots,l.
\end{equation}
That is, in the practical situation, the shapes of the underlying scatterers are known in advance in the sense that each component must be of a shape from a known admissible class. But we do not know the locations of those scatterer components and intend to recover them from a single wave detection. We would like to remark that it is not necessary to have $l=l'$. It may have $l>l'$, and in this case, there must be more than one component in $Q$ who has the same shape from $\mathscr{S}$; and it may also have $l<l'$, and in this case, there are less scatterers presented than the known admissible scatterers.

Next we introduce $l'$ indicator functions as follows,
\begin{equation}\label{eq:indicator regular}
I^k_r(z)=\frac{\bigg| \langle A(\theta;Q), e^{i\omega(\theta'-\theta)\cdot z} A(\theta; \Sigma_k)  \rangle_{T^2(\mathbb{S}^2)}  \bigg|}{\| A(\theta; \Sigma_k) \|^2_{T^2(\mathbb{S}^2)}},\quad k=1,2,\ldots, l',
\end{equation}
where $A(\theta;\Sigma_k)$ denotes the far-field pattern corresponding to the perfectly conducting obstacle $\Sigma_k$, $1\leq k\leq l'$.
In the following, we shall show that the $l'$ indictor functions introduced in \eqref{eq:indicator regular} can be used to locate the scatterer components $Q_j$ of $Q$. To that end, we shall make the following generic assumption that
\begin{equation}\label{eq:generic assumption}
A(\theta;\Sigma_k)\neq A(\theta; \Sigma_{k'}),\quad \mbox{for}\ \ k\neq k',\ 1\leq k, k'\leq l'.
\end{equation}
Assumption \eqref{eq:generic assumption} is closely related to a longstanding problem in the inverse scattering theory (cf. \cite{CCM,CK,LZr}): whether or not  can one uniquely determine an obstacle by a single wave measurement? That is, if two obstacles produce the same far-field data corresponding to a single incident plane wave, can one conclude that they must be the same? Though such uniqueness result is widespreadly believed to be true, it still remains open, particularly for an obstacle of regular size. The uniqueness is proved for obstacles of polyhedral type in \cite{L,LYZ}. Assumption~\eqref{eq:generic assumption} on the reference scatterers is of critical importance in our subsequently proposed locating scheme. Nonetheless, since the admissible class $\mathscr{S}$ is known, \eqref{eq:generic assumption} can be verified in advance.  Moreover, since $\mathscr{S}$ is known in advance, we can assume by reordering if necessary that
\begin{equation}\label{eq:assumption2}
\| A(\theta;\Sigma_{k}) \|_{T^2(\mathbb{S}^2)}\geq \| A(\theta;\Sigma_{k+1}) \|_{T^2(\mathbb{S}^2)},\quad k=1,2,\ldots,l'-1.
\end{equation}
That is, the sequence $\{\| A(\theta;\Sigma_{k}) \|_{T^2(\mathbb{S}^2)}\}_{k=1}^{l'}$ is nonincreasing. Next, we present a key theorem, which is the basis of our subsequent locating scheme.

\begin{thm}\label{thm:main2}
Let $Q$ be given in \eqref{eq:regular scatterer}, and the obstacle components are assumed to satisfy \eqref{eq:spare regular} and \eqref{eq:shape known}. The admissible reference scatterer space $\mathscr{S}$ is assumed to satisfy \eqref{eq:generic assumption} and \eqref{eq:assumption2}. Consider the indicator function $I_r^1$ introduced in \eqref{eq:indicator regular}. Suppose for some $j_0\in \{1,2,\ldots,l\}$, $M_{j_0}=\Sigma_1$, whereas $M_j\neq \Sigma_1$ for $j\in \{1,2,\ldots,l\}\backslash\{ j_0\}$. Then for each $z_j$, $j=1,2,\ldots,l$, there exists an open neighborhood of $z_j$, $neigh(z_j)$, such that
\begin{enumerate}
\item[(i).]~if $j=j_0$, then
\begin{equation}\label{eq:further ind 1}
\widetilde{I}_r^1(z):=|I_r^1(z)-1|\leq \mathcal{O}\left( \frac 1 L  \right ),\quad z\in neigh(z_{j_0}),
\end{equation}
and moreover, $z_{j_0}$ is a local minimum point for $\widetilde{I}_r^1(z)$;

\item[(ii).]~if $j\neq j_0$, then there exists $\epsilon_0>0$ such that
\begin{equation}\label{eq:further ind 2}
\widetilde{I}_r^1(z):=|I_r^1(z)-1|\geq \epsilon_0+\mathcal{O}\left( \frac 1 L \right ),\quad z\in neigh(z_j).
\end{equation}
\end{enumerate}
\end{thm}

By using Theorem~\ref{thm:main2}, our proposed locating scheme can be proceeded as follows, which shall be referred to as {\it Scheme R} in the rest of the paper.

\medskip

\noindent{\bf The second single-shot locating method: \textbf{Scheme R}}

\begin{enumerate}

\item[{\bf Step 1.}] For an unknown EM scatterer $Q$ in \eqref{eq:regular scatterer}, collect the far-field data by sending the detecting EM plane wave specified by  \eqref{eq:plane waves}.

\item[{\bf Step 2.}] Select a sampling region with a mesh $\mathcal{T}_h$ containing $Q$.

\item[{\bf Step 3.}] Collect in advance the far-field patterns associated with the admissible reference scatterer space $\mathscr{S}$
                      in \eqref{eq:reference space}, and reorder $\mathscr{S}$  if necessary to make it satisfy
\eqref{eq:assumption2}, and also verify the generic assumption \eqref{eq:generic assumption}.

\item[{\bf Step 4.}] Set $k=1$.

\item[{\bf Step 5.}] For each point $z\in \mathcal{T}_h$, calculate $I_r^k(z)$ (or $\widetilde{I}_r^k(z)=|I_r^k(z)-1|$).

\item[{\bf Step 6.}] Locate all those significant local maxima of $I_r^k(z)$ such that $I_r^k(z)\sim 1$ (or the minima of $\widetilde{I}_r^k(z)$ on $\mathcal{T}_h$ such that $\widetilde{I}_r^k(z)\ll 1$), where scatterer components of the form $z+\Sigma_k$ is located.

\item[{\bf Step 7.}] Trim all those $z+\Sigma_k$ found in {\bf Step 6} from $\mathcal{T}_h$.

\item[{\bf Step 8.}] If $\mathcal{T}_h=\emptyset$ or $k=l'$, then Stop; otherwise, set $k=k+1$, and go to {\bf Step 5}.

\end{enumerate}

It can be seen that our locating scheme $R$ progresses in a
recursive manner. For a scatterer $Q$ with multiple components, one
firstly locates the sub-components of shape $\Sigma_1$, which have
the most prominent scattering effect among all the scatterer
components. After locating all the sub-components of shape
$\Sigma_1$, one can exclude those components from the searching
region, and then repeats the same procedure to locate all the
sub-components of shape $\Sigma_2$, which, according to the ordering
\eqref{eq:assumption2}, have the most prominent scattering effect
among all the scattering components that still remain in the
searching region. Clearly, this procedure can be continued till one
locates all the scatter components. Theorem~\ref{thm:main2} remains
true for inhomogeneous medium scatterers if the generic condition
\eqref{eq:generic assumption} still holds and in that case, the
locating \textbf{Scheme R} could be extended to include
inhomogeneous medium components as well; see
Remark~\ref{rem:extension S} in the following for related
discussions.

\section{Proofs of Theorems~\ref{thm:main1} and \ref{thm:main2}}\label{sect:3}

This section is devoted to the proofs of Theorems~\ref{thm:main1} and \ref{thm:main2}, which are the theoretical cores for our locating \textbf{Schemes S} and \textbf{R}, respectively. We first derive two key lemmas.

\begin{lem}\label{lem:multiple scatterers}
Let
\begin{equation}\label{eq:multiple scatterers}
\Gamma=\bigcup_{j=1}^l \Gamma_j\quad \mbox{and}\quad (\Gamma;\varepsilon,\mu,\sigma)=\bigcup_{j=1}^l (\Gamma_j; \varepsilon_j,\mu_j,\sigma_j),
\end{equation}
be a scatterer with multiple components, where each $\Gamma_j$ is a bounded simply connected $C^2$ domain in $\mathbb{R}^3$. Assume that
\begin{equation}\label{eq:spare distributation}
L=\min_{1\leq j, j'\leq l, j\neq j'}\text{dist}(\Gamma_j, \Gamma_{j'})\gg 1.
\end{equation}
Then we have
\begin{equation}\label{eq:addition}
A(\theta;\Gamma)=\sum_{j=1}^l A(\theta; \Gamma_j)+\mathcal{O}(L^{-1}),
\end{equation}
where $A(\theta;\Gamma)$ and $A(\theta;\Gamma_j)$ denote the far-field patterns, respectively, corresponding to $(\Gamma;\varepsilon,\mu,\sigma)$ and $(\Gamma_j; \varepsilon_j,\mu_j,\sigma_j)$.
\end{lem}

\begin{proof}

We shall only consider a specific case with $l=2$, and $\Gamma_1$ is an obstacle component while $\Gamma_2$ is a medium component. Nevertheless, the general case can be proved following a completely similar manner. Moreover, we shall assume that $\omega$ is not an interior EM eigenvalue for $\Gamma_1$ (see Remark~\ref{rem:eigenvalue}).  For this case, the EM scattering corresponding to $\Gamma$ in \eqref{eq:multiple scatterers} is governed by the Maxwell system~\eqref{eq:Maxwell general} with $\Omega_o$ replaced by $\Gamma_1$ and $\Omega_m$ replaced by $(\Gamma_2;\varepsilon_2,\mu_2,\sigma_2)$.
We know that there exists a unique pair of solutions $E\in (H_{loc}^1(\mathbb{R}^3\backslash\overline{\Gamma}_1))^3$ and $H\in (H_{loc}^1(\mathbb{R}^3\backslash\overline{\Gamma}_1))^3$ to the Maxwell system (cf \cite{Ned}). Next, we shall make use of the integral equation method to prove the lemma. To that end, we let
\[
\Phi(x,y):=\frac{1}{4\pi}\frac{e^{i\omega|x-y|}}{|x-y|},\quad x, y\in\mathbb{R}^3, \ x\neq y,
\]
which is the fundamental solution to the differential operator $-\Delta-\omega^2$. Let ${R}\in\mathbb{R}_+$ be sufficiently large such that $\Gamma_1\cup\Gamma_2\Subset B_R$.

By the Stratton-Chu formula \cite{CK}, we have
\begin{align}
E(x)=& E^i(x)+\nabla_x\wedge\int_{\partial\Gamma_1}\nu(y)\wedge E(y)\Phi(x,y)\ ds_y \nonumber\\
&-\frac{1}{i\omega}\nabla_x\wedge\nabla_x\wedge\int_{\partial\Gamma_1}\nu(y)\wedge H(y)\Phi(x,y)\ ds_y\nonumber\\
&+\nabla_x\wedge\int_{\Gamma_2} (\nabla_y \wedge E(y)-i\omega H(y))\Phi(x,y)\ dy\nonumber\\
&-\nabla_x\int_{\Gamma_2} \nabla_y\cdot E(y)\Phi(x,y)\ d y\nonumber\\
&+i\omega\int_{\Gamma_2} (\nabla_y \wedge H(y)+i\omega E(y))\Phi(x,y)\ dy,\quad x\in B_R\backslash\overline{\Gamma}_1,\label{eq:fe}
\end{align}
and
\begin{align}
H(x)=& H^i(x)+\nabla_x\wedge\int_{\partial\Gamma_1} \nu(y)\wedge H(y)\Phi(x,y)\ ds_y\nonumber\\
&+\frac{1}{i\omega}\nabla_x\wedge\nabla_x\wedge\int_{\partial\Gamma_1} \nu(y)\wedge E(y) \Phi(x,y)\ ds_y\nonumber\\
&+\nabla_x\wedge\int_{\Gamma_2}(\nabla_y\wedge H(y)+i\omega E(y))\Phi(x,y)\ dy\nonumber\\
&-\nabla_x\int_{\Gamma_2}\nabla_y\cdot H(y)\Phi(x,y)\ dy\nonumber\\
&-i\omega\int_{\Gamma_2}(\nabla_y\wedge E(y)-i\omega H(y))\Phi(x,y)\ dy,\quad x\in B_R\backslash\overline{\Gamma}_1.\label{eq:fh}
\end{align}
By using the Maxwell equations \eqref{eq:Maxwell general}, and \eqref{eq:fe}--\eqref{eq:fh}, together with the use of the mapping properties of the layer potential operators (cf. \cite{CK,Ned}), we have the following system of integral equations
\begin{align}
E(x)=& E^i(x)-\frac{1}{i\omega}\nabla_x\wedge\nabla_x\wedge\int_{\partial\Gamma_1}\nu(y)\wedge H(y)\Phi(x,y)\ ds_y\nonumber\\
&+\nabla_x\wedge\int_{\Gamma_2}i\omega(\mu_2-1) H(y)\Phi(x,y)\ dy\nonumber\\
&+\int_{\Gamma_2}i\omega\left(1-\varepsilon_2-i\frac{\sigma_2}{\omega}\right) E(y)\Phi(x,y)\ dy,\qquad x\in\Gamma_2,\label{eq:syst1}\\
H(x)=& H^i(x)+\nabla_x\wedge\int_{\partial\Gamma_1}\nu(y)\wedge H(y)\Phi(x,y)\ ds_y\nonumber\\
&+\nabla_x\wedge\int_{\Gamma_2}i\omega\left(1-\varepsilon_2-i\frac{\sigma_2}{\omega}\right) E(y)\Phi(x,y)\ dy\nonumber\\
&-\int_{\Gamma_2}i\omega(\mu_2-1) H(y)\Phi(x,y)\ dy,\qquad x\in\Gamma_2,\label{eq:syst2}\\
\nu(x)\wedge H(x)=&2\nu(x)\wedge H^i(x)+2\nu(x)\wedge\nabla_x\wedge\int_{\partial\Gamma_1}\nu(y)\wedge H(y)\Phi(x,y)\ ds_y\nonumber\\
&+\nu(x)\wedge\nabla_x\wedge\int_{\Gamma_2}2i\omega\left(1-\varepsilon_2-i\frac{\sigma_2}{\omega}\right)E(y)\Phi(x,y)\ dy\nonumber\\
&-\nu(x)\wedge\int_{\Gamma_2} 2i\omega(\mu_2-1) H(y)\Phi(x,y)\ dy,\qquad x\in\partial\Gamma_1.\label{eq:syst3}
\end{align}
Next, we introduce the following volume integral operators
\begin{equation}\label{eq:integral operators}
\begin{split}
(\mathcal{L} E)(x)=& \int_{\Gamma_2}i\omega\left(1-\varepsilon_2-i\frac{\sigma_2}{\omega}\right) E(y)\Phi(x,y)\ dy, \quad x\in\Gamma_2,\\
(\mathcal{L}' E)(x)=& \nabla_x\wedge (\mathcal{L} E)(x),\quad x\in \Gamma_2,\\
(\mathcal{K} H)(x)=& \int_{\Gamma_2}i\omega(\mu_2-1) H(y)\Phi(x,y)\ dy,\quad x\in\Gamma_2,\\
(\mathcal{K}' H)(x)=& \nabla_x\wedge (\mathcal{K} H)(x),\quad x\in\Gamma_2,
\end{split}
\end{equation}
and the boundary integral operator
\begin{equation}\label{eq:integral operators 2}
(\mathcal{P} (\nu\wedge H))(x)= 2\nu(x)\wedge\nabla_x\wedge\int_{\partial\Gamma_1} \nu(y)\wedge H(y)\Phi(x,y)\ ds_y,\quad x\in\partial\Gamma_1.
\end{equation}
Two important function spaces are recalled as follows:
\begin{equation}\label{eq:function space}
\begin{split}
TH^{-1/2}(\partial\Gamma_1):=& \left\{\mathbf{a}\in (H^{-1/2}(\partial \Gamma_1))^3; \nu\wedge\mathbf{a}=0\quad\mbox{for a.e. $x\in\partial\Gamma_1$}\right\},\\
TH_{\text{Div}}^{-1/2}(\partial\Gamma_1):=& \left\{\mathbf{a}\in TH^{-1/2}(\partial \Gamma_1); \text{Div}(\mathbf{a})\in TH^{-1/2}(\partial \Gamma_1)\right\}.
\end{split}
\end{equation}
Furthermore, we let
\[
\begin{split}
\mathbf{b}_1(x):=&-\frac{1}{i\omega}\nabla_x\wedge\nabla_x\wedge\int_{\partial\Gamma_1}\nu(y)\wedge H(y)\Phi(x,y)\ ds_y,\quad x\in\Gamma_2,\\
\mathbf{b}_2(x):=& \nabla_x\wedge\int_{\partial\Gamma_1}\nu(y)\wedge H(y)\Phi(x,y)\ ds_y,\quad x\in\Gamma_2,\\
\mathbf{b}_3(x):=& \nu(x)\wedge\nabla_x\wedge\int_{\Gamma_2}2i\omega\left(1-\varepsilon_2-i\frac{\sigma_2}{\omega}\right)E(y)\Phi(x,y)\ dy\\
&-\nu(x)\wedge\int_{\Gamma_2} 2i\omega(\mu_2-1) H(y)\Phi(x,y)\ dy,\quad x\in\Gamma_1.
\end{split}
\]
Since $L=\mathrm{dist}(\Gamma_1,\Gamma_2)\gg 1$, one readily verifies that
\begin{equation}\label{eq:estt1}
\|\mathbf{b}_l\|_{L^2(\Gamma_2)^3}\leq \frac{C}{L}\|\nu\wedge H\|_{TH^{-1/2}_{\text{Div}}(\partial \Gamma_1)}, l=1,2,
\end{equation}
and
\begin{equation}\label{eq:estt2}
 \|\mathbf{b}_3\|_{TH^{-1/2}_{\text{Div}}(\Gamma_1)}\leq \frac{C}{L}\left(\|E\|_{L^2(\Gamma_2)^3}+\|H\|_{L^2(\Gamma_2)^3} \right ),
\end{equation}
where $C$ is a positive constant depending only on $\Gamma_1, \Gamma_2$ and $\omega$.

Next, by using the integral operators introduced in \eqref{eq:integral operators} and \eqref{eq:integral operators 2}, the integral equations \eqref{eq:syst1}--\eqref{eq:syst2} can be formulated as
\begin{equation}\label{eq:ii1}
\mathbf{A}:=\begin{pmatrix}
I-\mathcal{L} & \mathcal{K}' \\
-\mathcal{L}' &  I+\mathcal{K}
\end{pmatrix},\qquad \mathbf{A}\begin{pmatrix}
E(x)\\ H(x)
\end{pmatrix}-\begin{pmatrix}
\mathbf{b}_1(x)\\
\mathbf{b}_2(x)
\end{pmatrix}=\begin{pmatrix}
E^i(x)\\
H^i(x)
\end{pmatrix},\qquad x\in\Gamma_2,
\end{equation}
In a similar manner, \eqref{eq:syst3} can be formulated as
\begin{equation}\label{eq:ii2}
(I-\mathcal{P})(\nu\wedge H)(x)-\mathbf{b}_3(x)=2\nu(x)\wedge H^i(x),\quad x\in\Gamma_1.
\end{equation}
Referring to \cite{CK,Ned}, we know that both $\mathbf{A}: L^2(\Gamma_2)^3\wedge L^2(\Gamma_2)^3\rightarrow L^2(\Gamma_2)^3\wedge L^2(\Gamma_2)^3 $ and $I-\mathcal{P}: TH_{\text{Div}}^{-1/2}(\partial\Gamma_1)\rightarrow TH_{\text{Div}}^{-1/2}(\partial\Gamma_1)$ are Fredholm operators of index 0, and thus they are invertible. Using this fact and \eqref{eq:estt1}--\eqref{eq:estt2}, one obtains  from \eqref{eq:ii1} that
\begin{equation}\label{eq:soso1}
\begin{pmatrix}
E(x)\\ H(x)
\end{pmatrix}=\mathbf{A}^{-1}\begin{pmatrix}
E^i(x)\\ H^i(x)
\end{pmatrix}+\mathcal{O}(L^{-1}):=\begin{pmatrix}
\widetilde E(x)\\ \widetilde H(x)
\end{pmatrix}+\mathcal{O}(L^{-1}),\quad x\in\Gamma_2
\end{equation}
and
\begin{equation}\label{eq:soso2}
(\nu\wedge H)(x)=(I-\mathcal{P})^{-1}(2\nu\wedge H^i)(x)+\mathcal{O}(L^{-1}):=(\nu\wedge\widehat{H})(x)+\mathcal{O}(L^{-1}),\quad x\in\partial\Gamma_1,
\end{equation}
where $(\widetilde E, \widetilde H)$ are actually the EM fields corresponding to $(\Gamma_2;\varepsilon_2,\mu_2,\sigma_2)$, and $\widehat H$ is the magnetic field corresponding to $\Gamma_1$. Finally, by using \eqref{eq:soso1} and \eqref{eq:soso2}, and the integral representation \eqref{eq:syst1}, one readily has \eqref{eq:addition}.

The proof is complete.
\end{proof}

\begin{rem}\label{rem:eigenvalue}
If $\omega$ is an interior eigenvalue for $\Gamma_1$, one can overcome the problem by using the combined electric and magnetic dipole operators technique, and we refer to \cite[Chap. 6]{CK} for more details.

\end{rem}

\begin{lem}\label{lem:small expansion}
Let $(\Omega_1;\varepsilon_1,\mu_1,\sigma_1)$ be one component of $\Omega$ described in Section~\ref{sect:small} . Then we have
\begin{equation}\label{eq:small expansion 1}
\begin{split}
\!\!\!\!\!A(\theta;\Omega_1)=&e^{i\omega(\theta'-\theta)\cdot z_1}A(\theta; \rho D_1)\\
=& e^{i\omega(\theta'-\theta)\cdot z_1}\left[(\omega\rho)^3\left( \sum_{m=-1,0,1} a_{1,m} U_1^m(\theta)+ b_{1,m} V_1^m(\theta) \right)+\mathcal{O}((\omega\rho)^4)\right ],
\end{split}
\end{equation}
where $U_1^m$ and $V_1^m$ are the vectorial spherical harmonics introduced in \eqref{eq:orthonormal system}, and $a_{1,m}$ and $b_{1,m}$  $(m=-1,0,1)$, are constants depending only on $(D_1;\varepsilon_1, \mu_1, \sigma_1)$ and $p, \theta'$, but independent of $\omega\rho$.
\end{lem}

\begin{proof}
We first consider the case that $\sigma_1=+\infty$, namely, $\Omega_1$ is a perfectly conducting obstacle. It is directly verified that
\begin{equation}\label{eq:shift}
A(\theta;\Omega_1)=A(\theta;z_1+\rho D_1)=e^{i\omega(\theta'-\theta)\cdot z_1} A(\theta; \rho D_1).
\end{equation}
The EM scattering corresponding to the obstacle $\rho D_1$ is described by
\begin{equation}\label{eq:Maxwell small PEC}
\begin{cases}
\displaystyle{\nabla\wedge E_\rho-i\omega H_\rho=0}\ & \hspace*{-1cm} \mbox{in\ \ $\mathbb{R}^3\backslash\overline{\rho D_1}$},\\
\displaystyle{\nabla\wedge H_\rho+i\omega E_\rho=0}\ & \hspace*{-1cm} \mbox{in\ \ $\mathbb{R}^3\backslash\overline{\rho D_1}$},\\
E_\rho^+=(E_\rho-E^i)|_{\mathbb{R}^3\backslash\overline{\rho D_1}},\\
\nu\wedge E_\rho^+=-\nu\wedge E^i & \hspace*{-1cm} \mbox{on \ \ $\partial (\rho D_1)$},\\
\displaystyle{\lim_{|x|\rightarrow+\infty}|x|\left| (\nabla\wedge E_\rho^+)(x)\wedge\frac{x}{|x|}- i\omega E_\rho^+(x) \right|=0}.
\end{cases}
\end{equation}
Set
\begin{equation}\label{eq:scaling s}
\widetilde{E}(x):=E_\rho(\rho x),\quad \widetilde{H}(x):=H_\rho(\rho x),\quad \widetilde{E}^i(x)=E^i(\rho x)\quad\mbox{for\ \ $x\in\mathbb{R}^3\backslash\overline{D}_1$}.
\end{equation}
It is verified directly that
\begin{equation}\label{eq:Maxwell low frequency}
\begin{cases}
\displaystyle{\nabla\wedge \widetilde{E}-i\omega\rho \widetilde{H}=0}\ & \hspace*{-2cm} \mbox{in\ \ $\mathbb{R}^3\backslash\overline{D_1}$},\\
\displaystyle{\nabla\wedge \widetilde{H}+i\omega\rho \widetilde{E}=0}\ & \hspace*{-2cm} \mbox{in\ \ $\mathbb{R}^3\backslash\overline{D_1}$},\\
\widetilde{E}^+=(\widetilde{E}-\widetilde{E}^i)|_{\mathbb{R}^3\backslash\overline{D_1}},\\
\nu\wedge \widetilde{E}^+=-\nu\wedge \widetilde{E}^i & \hspace*{-2cm} \mbox{on \ \ $\partial D_1$},\\
\displaystyle{\lim_{|x|\rightarrow+\infty}|x|\left| (\nabla\wedge \widetilde{E}^+)(x)\wedge\frac{x}{|x|}- i\omega \widetilde{E}^+(x) \right|=0}.
\end{cases}
\end{equation}
Then by the low frequency asymptotics in \cite{DK}, one has
\begin{equation}\label{eq:lowf}
\widetilde{E}^+(x)=\frac{e^{i\omega\rho|x|}}{i\omega\rho|x|}\widetilde{A}(\theta)+\mathcal{O}\left(\frac{1}{|x|^2}\right),
\end{equation}
with
\begin{equation}\label{eq:low frequency 2}
\widetilde{A}(\theta)=\frac{(i\omega\rho)^3}{4\pi}[\theta\wedge(\theta\wedge\mathbf{a})-\theta\wedge\mathbf{b}]+\mathcal{O}((\omega\rho)^4),
\end{equation}
where $\mathbf{a}$ and $\mathbf{b}$ are two constant vectors, representing the electric and magnetic dipole moments, and they depend only on $p$, $\theta'$ 
and $D_1$, but independent of $\omega$ and $\rho$.
By \eqref{eq:scaling s}, one can readily see that
\begin{equation}\label{eq:low frequency 3}
A(\theta)=\frac{1}{i\omega}\widetilde{A}(\theta)=\frac{1}{i\omega}\frac{(i\omega\rho)^3}{4\pi}[\theta\wedge(\theta\wedge\mathbf{a})-\theta\wedge\mathbf{b}]+\mathcal{O}((\omega\rho)^4).
\end{equation}
Finally, by using the fact that $U_n^m$ and $V_n^m$ form a complete orthonormal system in $T^2(\mathbb{S}^2)$, it is straightforward to show that there exist $a_{1,m}$ and $b_{1,m}$, $m=-1,0,1$ such that
\[
\theta\wedge(\theta\wedge\mathbf{a})-\theta\wedge\mathbf{b}= \sum_{m=-1,0,1} a_{1,m} U_1^m(\theta)+ b_{1,m} V_1^m(\theta),
\]
which together with \eqref{eq:low frequency 2} and \eqref{eq:shift} implies \eqref{eq:small expansion 1}.

For the case when the underlying small scatterer is an inhomogeneous medium, by using a completely same scaling argument, together with the corresponding low frequency EM asymptotics in \cite{DK}, one can prove \eqref{eq:small expansion 1}.

The proof is completed.
\end{proof}

\begin{rem}
The low frequency EM asymptotics and the asymptotic expansions of EM fields due a small inclusion could also be found in \cite{AN,AVV,APRT,Mar}, where polarizability tensors are always involved. For the present study, we need asymptotic expansions in terms of the vectorial spherical harmonics.
\end{rem}

After technical preparations above, we are now in a position to show the proof of Theorem~\ref{thm:main1}.

\begin{proof}[Proof of Theorem~\ref{thm:main1}]
First, by Lemmas~\ref{lem:multiple scatterers} and \ref{lem:small expansion}, we have
\begin{equation}\label{eq:pa1}
\begin{split}
&A(\theta;\Omega)=\sum_{j=1}^l A(\theta; \Omega_l)+\mathcal{O}\left(\frac 1 L\right)\\
=& \sum_{j=1}^l e^{i\omega(\theta'-\theta)\cdot z_j}\left[(\omega\rho)^3 A^j(\theta)+\mathcal{O}((\omega\rho)^4)\right ]+\mathcal{O}\left(\frac 1 L\right)\\
=& \sum_{j=1}^l e^{i\omega(\theta'-\theta)\cdot z_j}\left[(\omega\rho)^3\left( \sum_{m=-1,0,1} a_{1,m}^{(j)} U_1^m(\theta)+ b_{1,m}^{(j)} V_1^m(\theta) \right)+\mathcal{O}((\omega\rho)^4)\right ]\\
&\qquad +\mathcal{O}\left(\frac 1 L\right)
\end{split}
\end{equation}
where $a_{1,m}^{(j)}$ and $b_{1,m}^{(j)}$, $m=-1,0,1$, are constants dependent only on $(D_j;\varepsilon_j, \mu_j, \sigma_j)$ and $p, \theta'$. In \eqref{eq:pa1}, we have introduced
\[
A^j(\theta):=\sum_{m=-1,0,1} \left( a_{1,m}^{(j)} U_1^m(\theta)+ b_{1,m}^{(j)} V_1^m(\theta) \right).
\]

Next, without loss of generality, we only consider the indicating behaviors of $I_s(z)$ in a small open neighborhood of $z_1$, i.e., $z\in neigh(z_1)$. Clearly, we have
\begin{equation}
\omega|z_j-z|\geq \omega L\gg 1\quad \mbox{for $z\in neigh(z_1)$ and $j=2,3,\ldots, l$}.
\end{equation}
Hence, by the Riemann-Lebesgue lemma and \eqref{eq:pa1}, we have
\begin{equation}\label{eq:ss2}
\begin{split}
&\langle A(\theta;\Omega), e^{i\omega (\theta'-\theta)\cdot z} U_1^m(\theta)  \rangle_{T^2(\mathbb{S}^2)} \\
=&(\omega\rho)^3 \langle e^{i\omega(\theta'-\theta)\cdot z_1}A^1(\theta), e^{i\omega(\theta'-\theta)\cdot z} U_1^m(\theta)\rangle_{T^2(\mathbb{S}^2)}+\mathcal{O}\left((\omega\rho)^3(\frac 1 L+\rho)\right)\\
\leq & (\omega\rho)^3\left( a_{1,m}^{(1)}+\mathcal{O}\left(\frac 1 L+\rho\right) \right ),\quad m=-1,0,1.
\end{split}
\end{equation}
In \eqref{eq:ss2}, we have employed the orthogonality of vector spherical harmonics and also the Cauchy-Schwartz inequality. Moreover, by the Cauchy-Schwartz inequality, we know the strict inequality would hold in the last inequality of \eqref{eq:ss2} if $z\neq z_1$, and only when $z=z_1$, the equality would hold. In a completely similar manner, one can show that
\begin{equation}\label{eq:ss3}
\begin{split}
&\langle A(\theta;\Omega), e^{i\omega (\theta'-\theta)\cdot z} V_1^m(\theta)  \rangle_{T^2(\mathbb{S}^2)} \\
=&(\omega\rho)^3 \langle e^{i\omega(\theta'-\theta)\cdot z_1}A^1(\theta), e^{i\omega(\theta'-\theta)\cdot z} V_1^m(\theta)\rangle_{T^2(\mathbb{S}^2)}+\mathcal{O}\left((\omega\rho)^3(\frac 1 L+\rho)\right)\\
\leq & (\omega\rho)^3\left( b_{1,m}^{(1)}+\mathcal{O}\left(\frac 1 L+\rho\right) \right ),\quad m=-1,0,1,
\end{split}
\end{equation}
where strict inequality would hold for the last relation if $z\neq z_1$, and only when $z=z_1$, the equality would hold.

Hence, by \eqref{eq:ss2} and \eqref{eq:ss3}, we have
\begin{equation}\label{eq:ss4}
\begin{split}
&\sum_{m=-1,0,1} \left(  {\bigg|\left\langle A(\theta;\Omega), e^{i\omega (\theta'-\theta)\cdot z}\, U_1^m(\theta)  \right\rangle_{T^2(\mathbb{S}^2)}\bigg|^2} \right.  \\
&\qquad\qquad\qquad +\left. {\left|\left\langle A(\theta;\Omega), e^{i\omega(\theta'-\theta)\cdot z}\, V_1^m(\theta)  \right\rangle_{T^2(\mathbb{S}^2)}\right|^2} \right)  \\
&\leq (\omega\rho)^6\left( \sum_{m=-1,0,1} |a_{1,m}^{(1)}|^2+|b_{1,m}^{(1)}|^2+ \mathcal{O}\left(\frac 1 L+\rho\right)  \right ),
\end{split}
\end{equation}
where the equality would hold only when $z=z_1$. On the other hand, by using \eqref{eq:pa1}, it is straightforward to show that
\begin{equation}\label{eq:ss5}
\|A(\theta;\Omega)\|^2_{T^2(\mathbb{S}^2)}=(\omega\rho)^6\sum_{j=1}^l\left( \sum_{m=-1,0,1} |a_{1,m}^{(j)}|^2+|b_{1,m}^{(j)}|^2+ \mathcal{O}\left(\frac 1 L+\rho\right)  \right ).
\end{equation}
By \eqref{eq:ss4} and \eqref{eq:ss5}, we see that for $z\in neigh(z_1)$,
\begin{equation}\label{eq:ss6}
\begin{split}
I_s(z)=\frac{1}{\|A(\theta;\Omega)\|^2_{T^2(\mathbb{S}^2)}}\sum_{m=-1,0,1}\bigg( & {\bigg|\left\langle A(\theta;\Omega), e^{i\omega (\theta'-\theta)\cdot z}\, U_1^m(\theta)  \right\rangle_{T^2(\mathbb{S}^2)}\bigg|^2}\\
& +{\bigg|\left\langle A(\theta;\Omega), e^{i\omega(\theta'-\theta)\cdot z}\, V_1^m(\theta)  \right\rangle_{T^2(\mathbb{S}^2)}\bigg|^2}         \bigg)\\
\leq \frac{\sum\limits_{m=-1}^1 |a_{1,m}^{(1)}|^2+|b_{1,m}^{(1)}|^2}{\sum\limits_{j=1}^l \sum\limits_{m=-1}^1 |a_{1,m}^{(j)}|^2+|b_{1,m}^{(j)}|^2}+&\mathcal{O}\left(\frac 1 L+\rho\right),
\end{split}
\end{equation}
where only when $z=z_1$, the equality would hold in the last relation.  Set
\[
K_0^1:=\frac{\sum\limits_{m=-1}^1 |a_{1,m}^{(1)}|^2+|b_{1,m}^{(1)}|^2}{\sum\limits_{j=1}^l \sum\limits_{m=-1}^1 |a_{1,m}^{(j)}|^2+|b_{1,m}^{(j)}|^2}.
\]
Using \eqref{eq:small expansion 1} and \eqref{eq:ss5}, it is readily seen that
\[
K^1:=\frac{\|A(\theta;\Omega_1)\|^2_{T^2(\mathbb{S}^2)}}{\|A(\theta;\Omega)\|^2_{T^2(\mathbb{S}^2)}}=K_0^1+\mathcal{O}\left(\frac 1 L+\rho\right).
\]

Thus the proof is completed.
\end{proof}

Next, we show the proof of Theorem~\ref{thm:main2}.

\begin{proof}[Proof of Theorem~\ref{thm:main2}]~First, by Lemma~\ref{lem:multiple scatterers}, we have
\begin{equation}\label{eq:m21}
\begin{split}
A(\theta;\Omega)& =A(\theta; \bigcup_{j=1}^l Q_j)=\sum_{j=1}^l A(\theta; Q_j)+\mathcal{O}\left(\frac 1 L\right )\\
=&\sum_{j=1}^l A(\theta;\bigcup_{j=1}^l z_j+M_j)+\mathcal{O}\left(\frac 1 L\right )\\
=&\sum_{j=1}^l e^{i\omega(\theta'-\theta)\cdot z_j} A(\theta; M_j)+\mathcal{O}\left(\frac 1 L\right ).
\end{split}
\end{equation}
Let us consider the indicator function $I_r^1(z)$ in \eqref{eq:indicator regular}. Without loss of generality, we assume that $M_1=\Sigma_1$ and $M_j\neq\Sigma_1$ for $j=2,\ldots,l$. Let $z\in neigh (z_1)$. By \eqref{eq:spare regular}, \eqref{eq:m21} and the Riemannian-Lebesgue lemma, we have
\begin{equation}\label{eq:m22}
\begin{split}
&\left|\langle A(\theta; \Omega), e^{i\omega(\theta'-\theta)\cdot z}A(\theta;\Sigma_1) \rangle_{T^2(\mathbb{S}^2)}\right|\\
=&\left|\langle e^{i\omega(\theta'-\theta)\cdot z_1}A(\theta; \Sigma_1), e^{i\omega(\theta'-\theta)\cdot z}A(\theta;\Sigma_1) \rangle_{T^2(\mathbb{S}^2)}\right|+\mathcal{O}\left(\frac 1 L\right )\\
\leq & \|A(\theta;\Sigma_1)\|_{T^2(\mathbb{S}^2)}^2+\mathcal{O}\left(\frac 1 L\right ),
\end{split}
\end{equation}
where in the last relation, the equality would hold only when $z=z_1$. Hence, for a sufficiently small neighborhood, $neigh(z_1)$,
\[
|I_r(z)-1|\leq \mathcal{O}\left(\frac 1 L\right )\quad \mbox{for\ \ $z\in neigh(z_1)$}.
\]
Next, we let $z\in neigh(z_2)$, with $neigh(z_2)$ sufficiently small. Then, again by \eqref{eq:spare regular}, \eqref{eq:m21} and the Riemannian-Lebesgue lemma, we have
\begin{equation}\label{eq:m23}
\begin{split}
&\left|\langle A(\theta; \Omega), e^{i\omega(\theta'-\theta)\cdot z_2}A(\theta;\Sigma_1) \rangle_{T^2(\mathbb{S}^2)}\right|\\
=&\left|\langle e^{i\omega(\theta'-\theta)\cdot z_2}A(\theta; M_2), e^{i\omega(\theta'-\theta)\cdot z_2}A(\theta;\Sigma_1) \rangle_{T^2(\mathbb{S}^2)}\right|+\mathcal{O}\left(\frac 1 L\right )\\
< & \|A(\theta;M_2)\|_{T^2(\mathbb{S}^2)} \|A(\theta;\Sigma_1)\|_{T^2(\mathbb{S}^2)}+\mathcal{O}\left(\frac 1 L\right ),
\end{split}
\end{equation}
where in the last relation, strict inequality holds in light of the assumption \eqref{eq:generic assumption} and the fact $M_2\neq \Sigma_1$. Hence, by \eqref{eq:m23} and \eqref{eq:assumption2}, it is readily shown that
\begin{equation}\label{eq:m33}
I_r(z_2)<\frac{\|A(\theta; M_2)\|_{T^2(\mathbb{S}^2)}}{\|A(\theta;\Sigma_1)\|_{T^2(\mathbb{S}^2)}}\leq 1+\mathcal{O}\left(\frac 1 L\right ).
\end{equation}
Hence, there  exists some $\epsilon_0\in\mathbb{R}_+$ such that
\[
\widetilde{I}_r^1(z)=|I_r^1(z)-1|\geq \epsilon_0+\mathcal{O}\left( \frac 1 L \right )\quad\mbox{for\ $z\in neigh(z_2)$}.
\]
In a completely similar manner, one can show the indicating behaviors of $\widetilde{I}_r^1(z)$ for $z\in neigh(z_j)$, $j=3,\ldots,l$.

The proof is completed.
\end{proof}

\begin{rem}\label{rem:extension S}
Through our proof of Theorem~\ref{thm:main2}, one can see that our locating \textbf{Scheme R} could be extended to a more general setting by including inhomogeneous medium components as follows. Suppose in $\mathscr{S}$, some reference scatterer, say, $(\Sigma_j; \varepsilon_j, \mu_j, \sigma_j)$, $1\leq j\leq l'\leq l$, are inhomogeneous media, with the EM parameters $\varepsilon_j, \mu_j$ and $\sigma_j$ known as well. Let $Q$ be a scatterer with multiple components, such that each component is a translation of some reference scatterer. Then our locating \textbf{Scheme R}  could be extended to this more general setting provided the reference scatterer space $\mathscr{S}$ satisfies the generic assumption \eqref{eq:generic assumption}. However, to our best knowledge, there is no such uniqueness result in the literature. On the other hand, as we remarked earlier that $\mathscr{S}$ is given in advance, one coud verify \eqref{eq:generic assumption} in advance, and if it is satisfied, our locating scheme could apply.
\end{rem}

\section{Numerical experiments  and discussions \label{sec:numerics}}

In this section, we have carried out a series of numerical experiments for different benchmark problems to test the performance
of our proposed locating \textbf{Schemes S} and \textbf{R}. The results achieved are
consistent with our theoretical predictions in Sections~\ref{sect:2} and \ref{sect:3} in a sound manner.
Besides, the numerical results reveal some very promising features of  the imaging schemes that were not covered in our theoretical analysis.


We first briefly describe our experimental settings. Let $e_1=(1,\,0,\,0)^{T}$, $e_2=(0,\,1,\,0)^{T}$ and
$e_3=(0,\,0,\,1)^{T}$ be the three canonical Cartesian bases. The single detecting/incident  wave we shall employ for our numerical examples is the
plane wave specified by \eqref{eq:plane waves} with the polarization $p=e_3$ and impinging direction $\theta'=e_1$. Moreover, We shall
take the unitary wavelength $\lambda=1$, namely the frequency
$\omega=2\pi$.
In all
the examples, the electric far-field pattern $A$ is observed at 590
Lebedev quadrature points distributed on the unit sphere
$\mathbb{S}^{2}$ (cf.~\cite{Leb99} and references therein). The exact far-field data
$A(\theta)$ are corrupted point-wise by the formula
\begin{equation}
A_{\delta}(\theta)=A(\theta)+\delta\zeta_1\underset{\theta}{\max}|A(\theta)|\exp(i2\pi
\zeta_2)\,,
\end{equation}
where $\delta$ refers to the relative noise level, and both  $\zeta_1$ and $\zeta_2$ follow the
uniform distribution ranging from $-1$ to $1$. The scattered
electromagnetic fields are synthesized using the quadratic edge 
element discretization in the spherical domain centered at the origin with radius $4\lambda$ enclosed by a
spherical PML layer of width $\lambda$ to damp the reflection. Local adaptive
refinement techniques within the inhomogeneous scatterer are adopted to
enhance the resolution of the scattered field. The far-field data
are approximated by the integral equation representation \cite[p.~181,
Theorem~3.1]{PiS02} on the sphere centered at the origin with radius
$3.5$ using the numerical quadrature. We refine the mesh successively
till the relative maximum error of successive groups of far-field
data is below $0.1\%$. The far-field patterns on the finest mesh
are used as the exact data. The values of the indicator functions have  been
normalized between $0$ and $1$ to highlight the positions
identified. The sampling domain is fixed to be
$\mathcal{T}=[-2\lambda,\,2\lambda]^{3}$, which is then divided into small
cubes of equal width $h=0.01\lambda$ yielding the sampling mesh $\mathcal{T}_h$. The
orthogonal  slices of the contours of the
indicator function values will be displayed as an estimate to the
profiles of the unknown scatterers.  In the sequel, for brevity, we shall refer to
the newly proposed single-shot locating  methods of \textbf{Schemes S} and \textbf{R} as the {\bf SSM(s)} and
{\bf SSM(r)}, respectively.

The scatterers under concern include a cube, a ball
of different radii, and revolving solids from a kite  and a peanut  parameterized in the $x-y$ plane as follows \cite{LLZ09}:
\begin{eqnarray*}
  (x(t),y(t),0) &:=& (\cos t +0.65 \cos 2t - 0.65, 1.5\sin t, 0), \quad 0 \le t \le 2\pi\,,  \\
  (x(t),y(t),0) &:=& \sqrt{3\cos^2 t +1}(\cos t, \sin t,0), \quad 0 \le t \le 2\pi\,.
\end{eqnarray*}

In the following, three groups of experiments shall be conducted. The first group of experiments is on locating  point-like small-sized scatterers in various scenarios by the {\bf SSM(s)}, and the second group of  experiments is on testing the {\bf SSM(r)} for locating regular-sized scatterers. In the third group of experiments, we shall test the performance of {\bf SSM(s)} on imaging `partially-small' line-segment-like scatterers.

\subsection{The SSM(s) for point-like scatterers}


\smallskip{}

\textbf{Example 1.}~In this example, we
consider a cube scatterer of length $0.02\lambda$ located at the
origin, with the EM parameters given by $\varepsilon=4, \mu=1$ and $\sigma=0$. The orthogonal slices of the contours of the indicator
function $I_s(z)$ in \eqref{eq:indicator function} for the \textbf{SSM(s)} are given in
Fig.~\ref{fig:ex1}. It can be readily seen that the {\bf SSM(s)} can locate the small scatterer in a very accurate and stable manner. Indeed, even $20\%$ random noise is attached to the measurement data, the {\bf SSM(s)} still yields a very robust and accurate locating.


\begin{figure}
\hfill{}\includegraphics[width=0.4\textwidth]{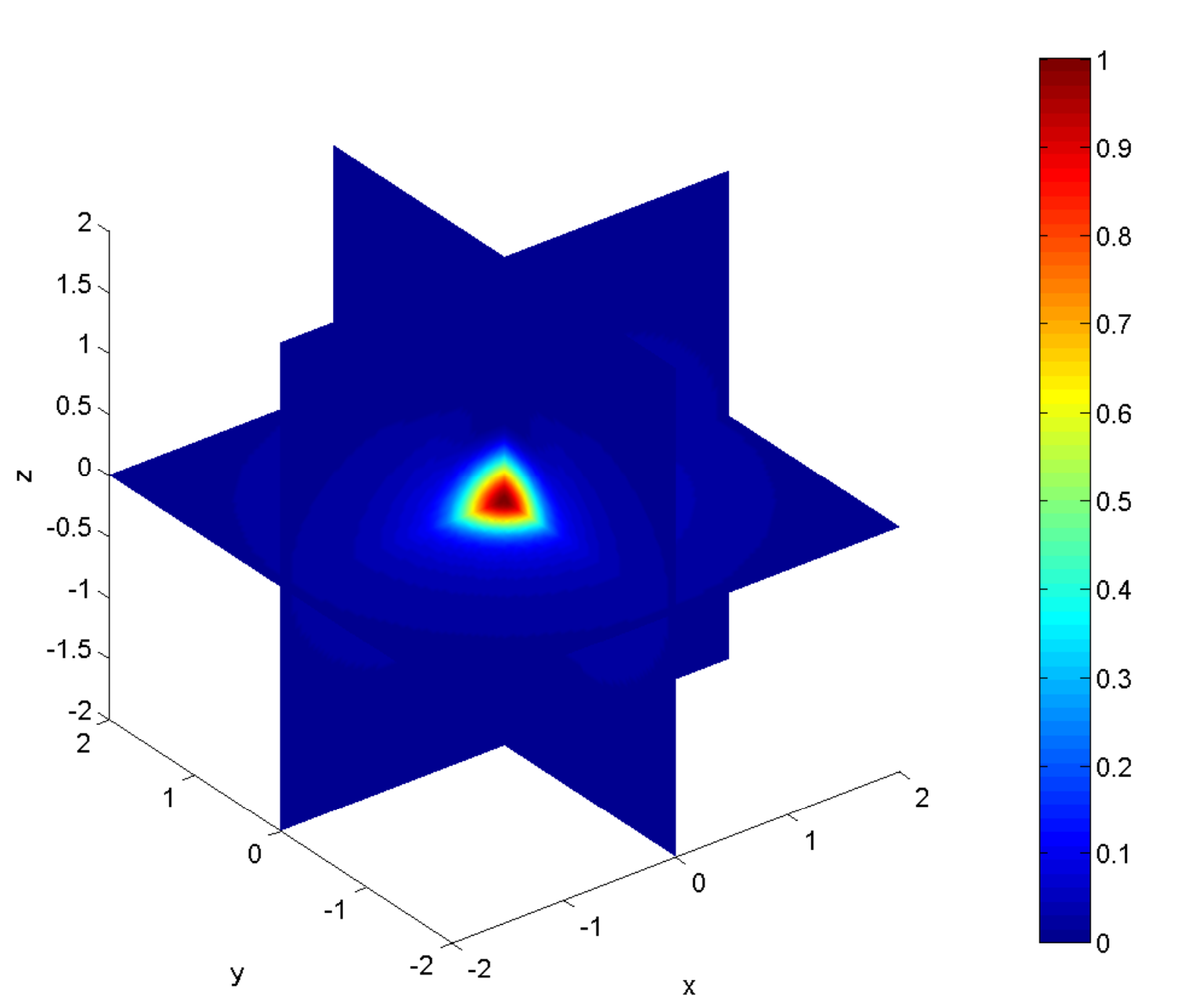}
\hfill{}\includegraphics[width=0.4\textwidth]{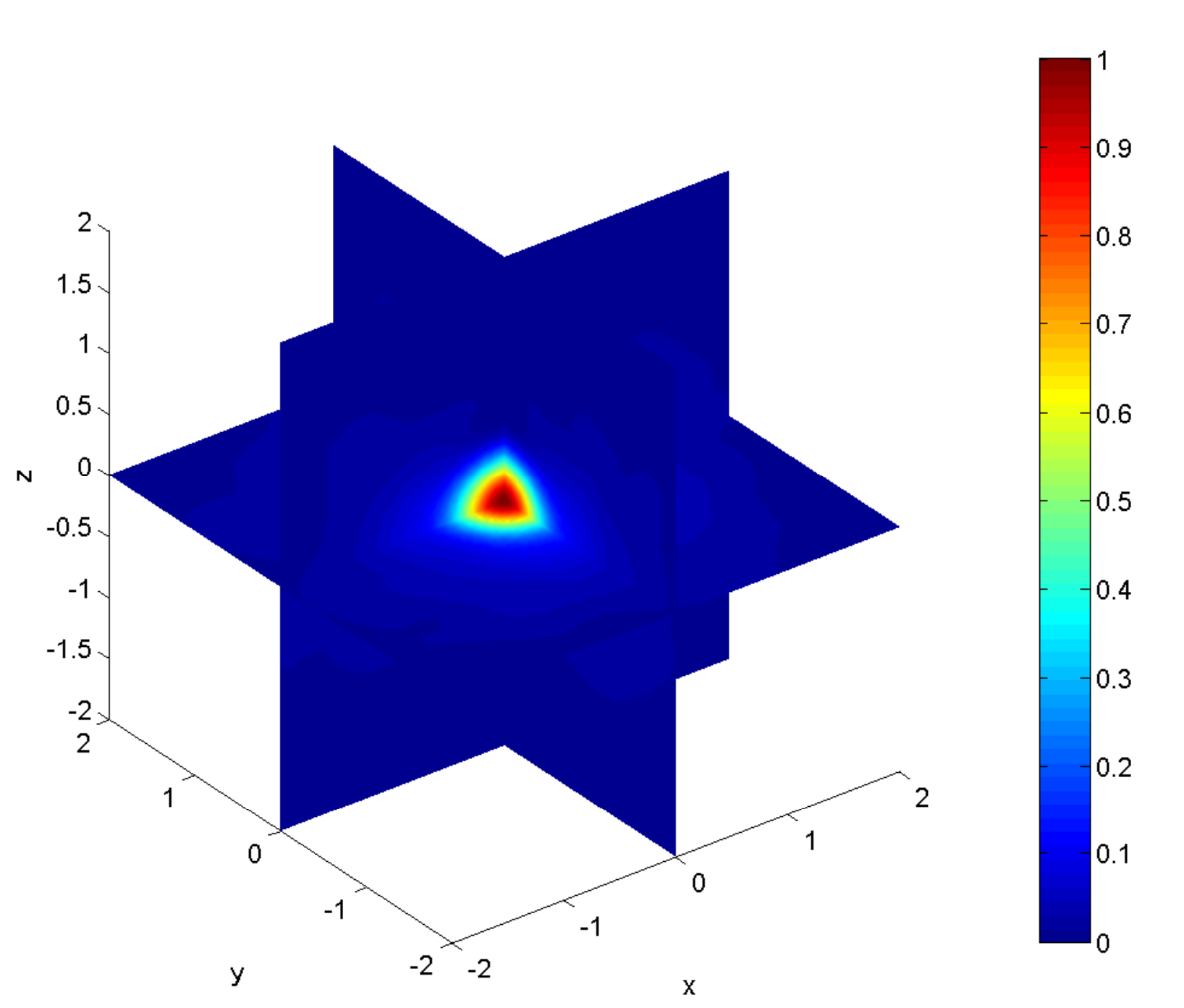}\hfill{}

\caption{\label{fig:ex1} Reconstruction results for Example 1:
(Left) Exact far-field data; (Right) Noisy far-field data with
$\delta=20\%$.}
\end{figure}

%

\medskip

\noindent\textbf{Example 2.}~In this example, we consider a scatterer consisting of a ball medium component of radius $0.2\lambda$ positioned at $(1.5\lambda,\,1.5\lambda,\,0)$ and a ball PEC obstacle of radius $0.2\lambda$ positioned at
$(-1.5\lambda,\,-1.5\lambda,\,0)$. The EM parameters of the first scatterer 
component are taken to the same as the one in Example~1. For this example, the orthogonal slices of
the contours of the indicator function $I_s(z)$  for the \textbf{SSM(s)}  are shown in
Fig.~\ref{fig:ex2}. 
The \textbf{SSM(s)} yields a very accurate identification of the
location of both scatterers even if the measurement data is
significantly perturbed to a high-level noise. 
This example demonstrates that {\bf SSM(s)} can locate the multiple scatterer components without knowing the physical property of each component in advance. Moreover, we would like to note a promising feature of the {\bf SSM(s)}: if one chooses a cut-off value to be $0.7$ to separate the region where $I_s(z)$ is bigger than the cut-off value, then the rough profiles of the scatterer components would appear. Hence, in addition to find the locations of the scatterer components, our proposed scheme could also be able to qualitatively image the supports/shapes of the unknown scatterers.



\begin{figure}
\hfill{}\includegraphics[width=0.4\textwidth]{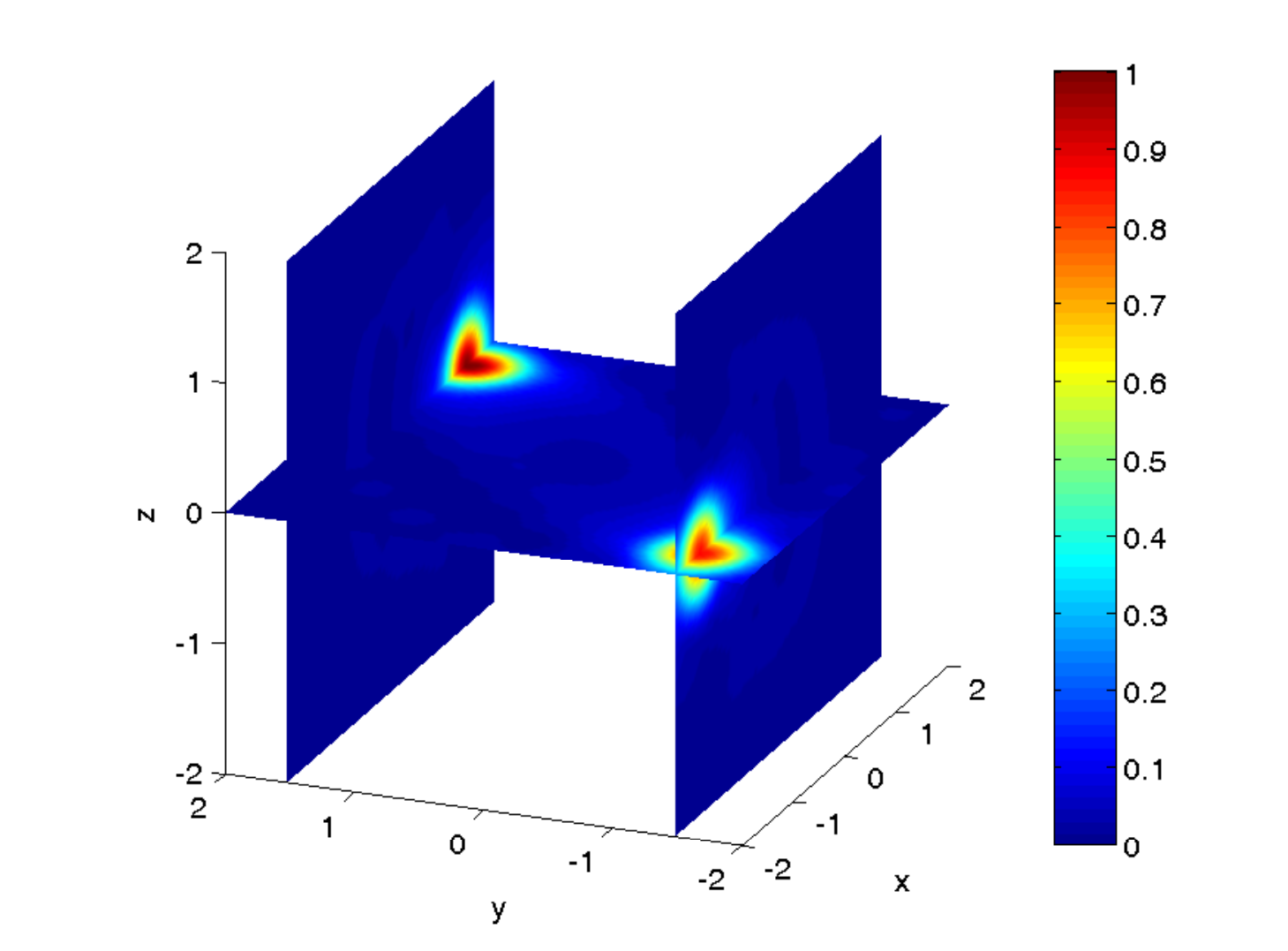}
\hfill{}\includegraphics[width=0.4\textwidth]{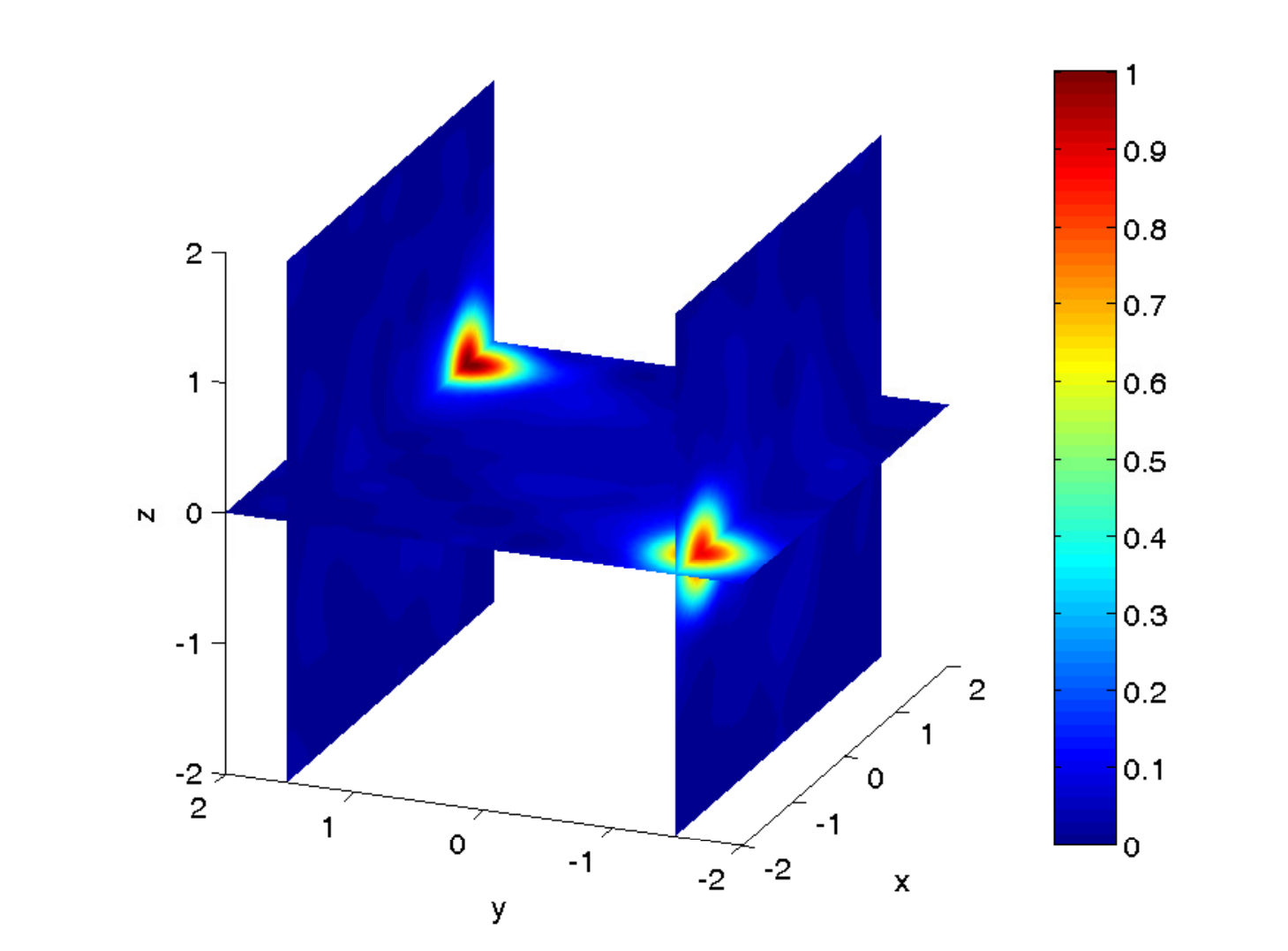}\hfill{}

\caption{\label{fig:ex2} Reconstruction results for Example 2:
(Left) Exact far-field data; (Right) Noisy far-field data with
$\delta=20\%$.}
\end{figure}

\medskip

\noindent\textbf{Example 3.}~In this example, three scatterers are presented at $(-\lambda,\,-\lambda,\,\lambda)$, $(\lambda,\,\lambda,\,0)$,
$(\lambda,\,\lambda,\,-\lambda)$, respectively. The first two scatterers are balls of radius $0.2\lambda$, with the EM parameters given by $\varepsilon_1=4+\sin x_1$, $\mu_1=1$, $\sigma_1=0$ and $\varepsilon_2=4+\cos x_2$, $\mu_2=1$, $\sigma_2=0$, respectively. Here, we have made use of $x=(x_1,x_2,x_3)$ to denote a point in $\mathbb{R}^3$. The third one is a non-convex kite-shaped revolving scatterer scaled by a relative size $\rho = 0.2$ with the EM parameters given by $\varepsilon_3=4+ x_3$, $\mu_3=1$ and $\sigma_3=0$. The numerical reconstruction results are shown in Fig.~\ref{fig:ex3}.
It can be seen from Fig.~\ref{fig:ex3} that  the {\bf SSM(s)} is capable of locating multiple scatterer components with variable contents. Moreover, it verifies that the small scatterers are not necessarily convex as assumed in our theoretical justification (cf. the remark below Eq. \eqref{eq:qualitative assumptions}).


\begin{figure}
\hfill{}\includegraphics[width=0.4\textwidth]{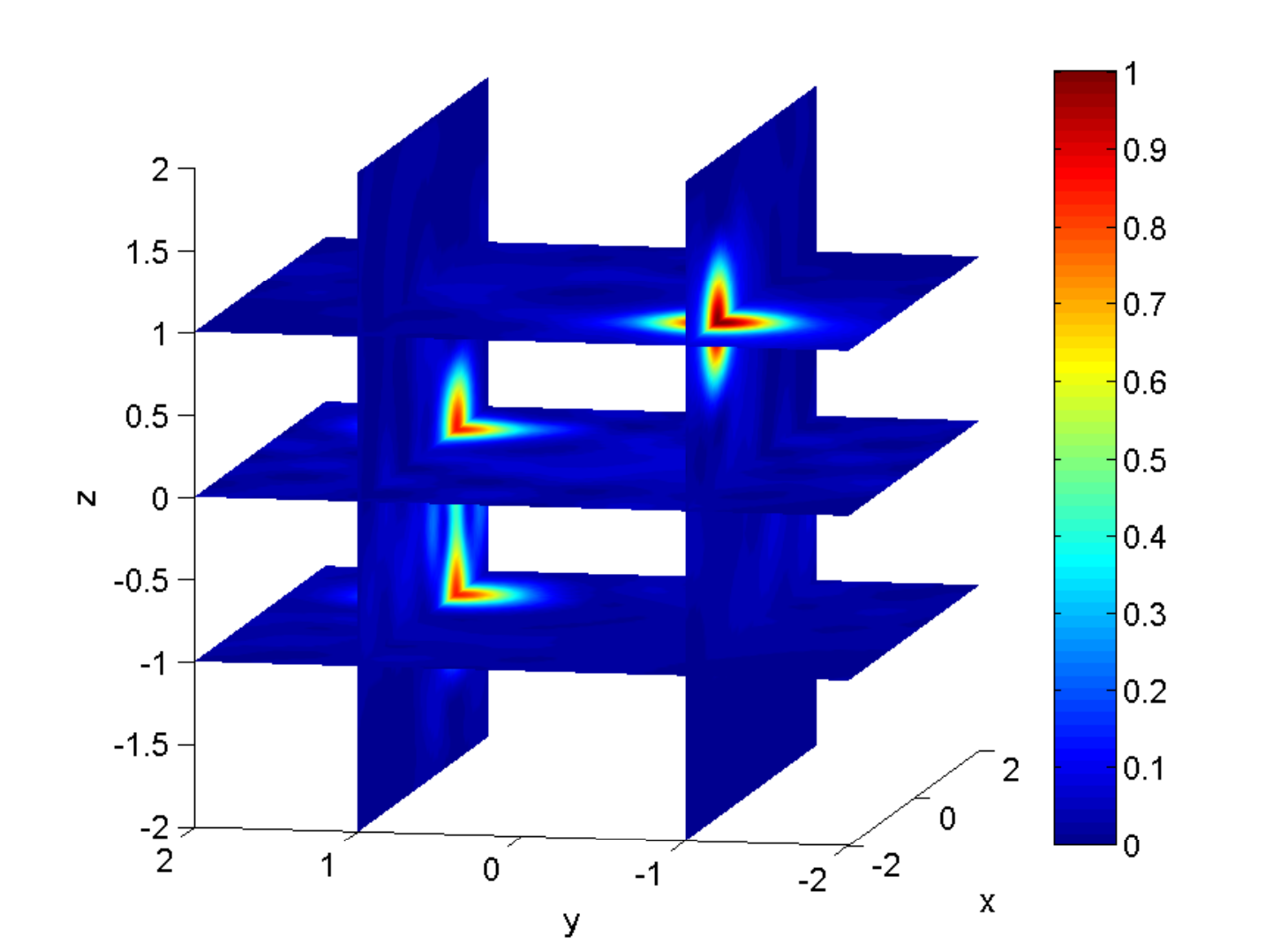}
\hfill{}\includegraphics[width=0.4\textwidth]{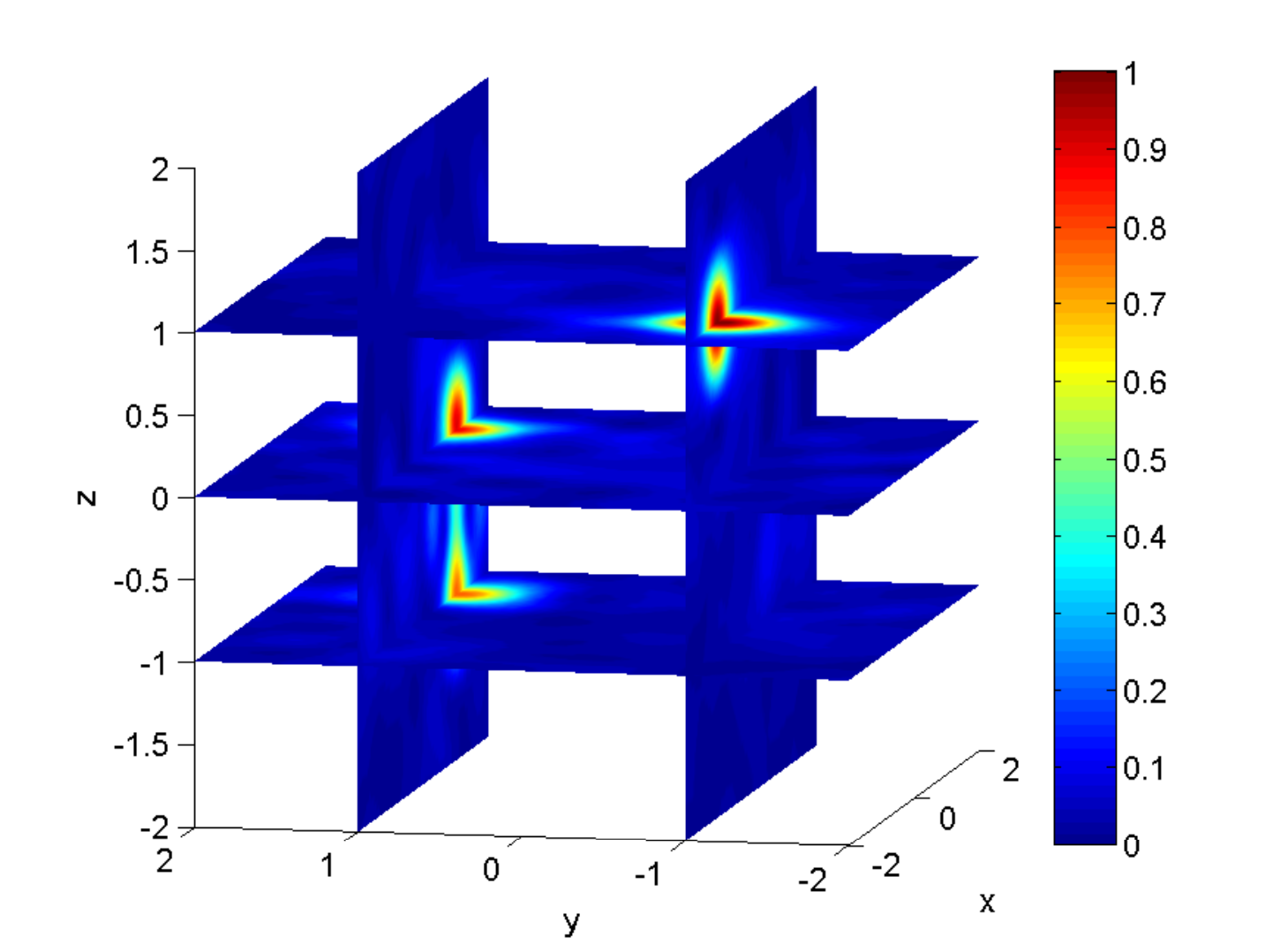}\hfill{}

\caption{\label{fig:ex3} Reconstruction results for Example 3:
(Left) Exact far-field data; (Right) Noisy far-field data with
$\delta=20\%$.}
\end{figure}

%

\medskip

\noindent\textbf{Example 4.}~We consider two ball
scatterers of radius $0.2\lambda$, located at
$(-0.45\lambda,\,0,\,0)$ and $(0.45\lambda,\,0,\,0)$, respectively, with the same EM parameters $\varepsilon=4, \mu=1$ and $\sigma=0$. We shall investigate the lower distance limit between the underlying separate scatterer components for the {\bf SSM(s)}. The results are
shown in Fig.~$\text{\,}$\ref{fig:ex4}. It can be seen in this case, namely the distance between the two components is of a half wavelength, the {\bf SSM(s)} can locate both scatterer components and separate them well. If we further reduce the distance between the two components (less than a half wavelength), the {\bf SSM(s)} can no longer separate the two scatterer components, though it can still roughly locate them.

\begin{figure}
\hfill{}\includegraphics[width=0.4\textwidth]{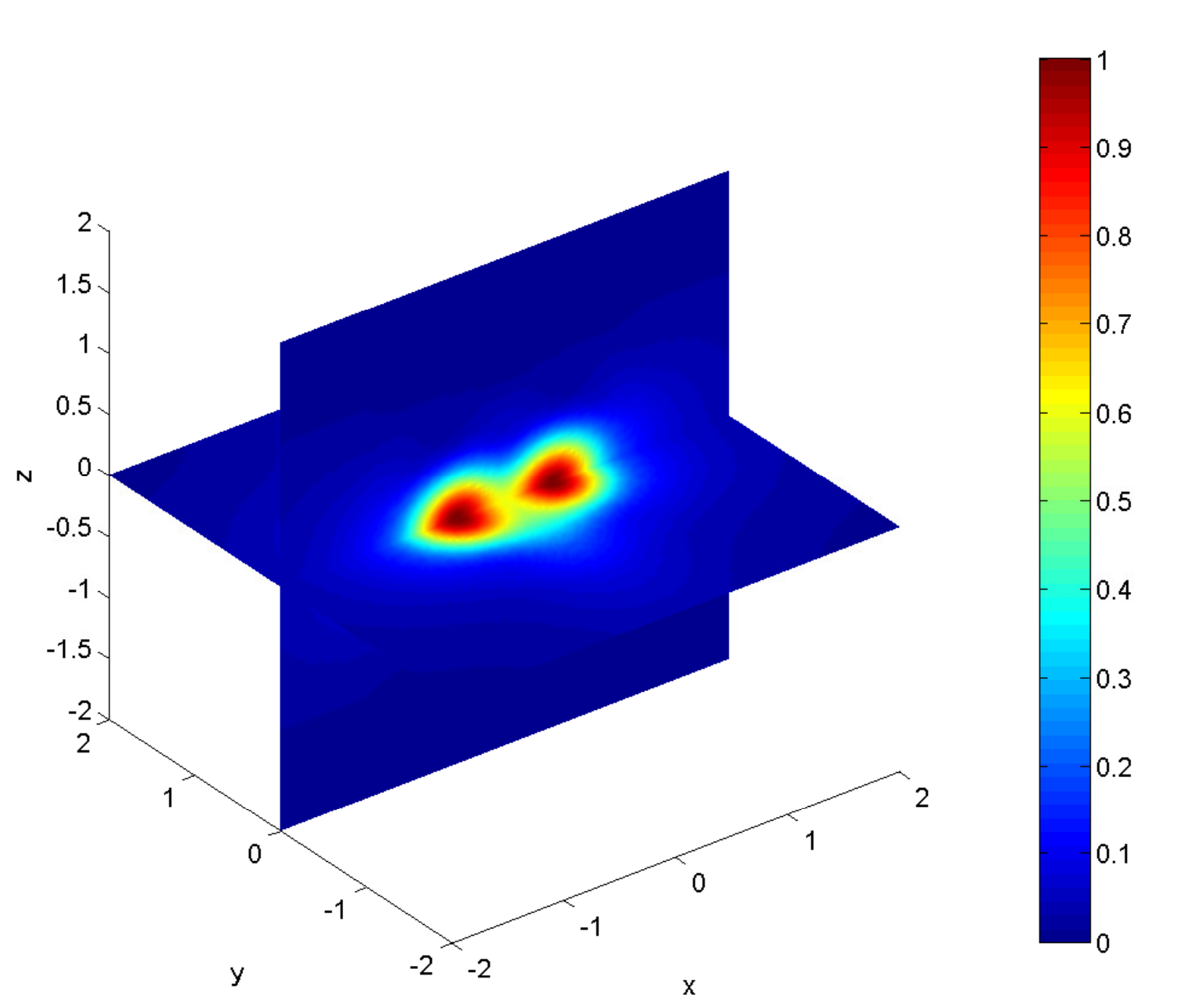}
\hfill{}\includegraphics[width=0.4\textwidth]{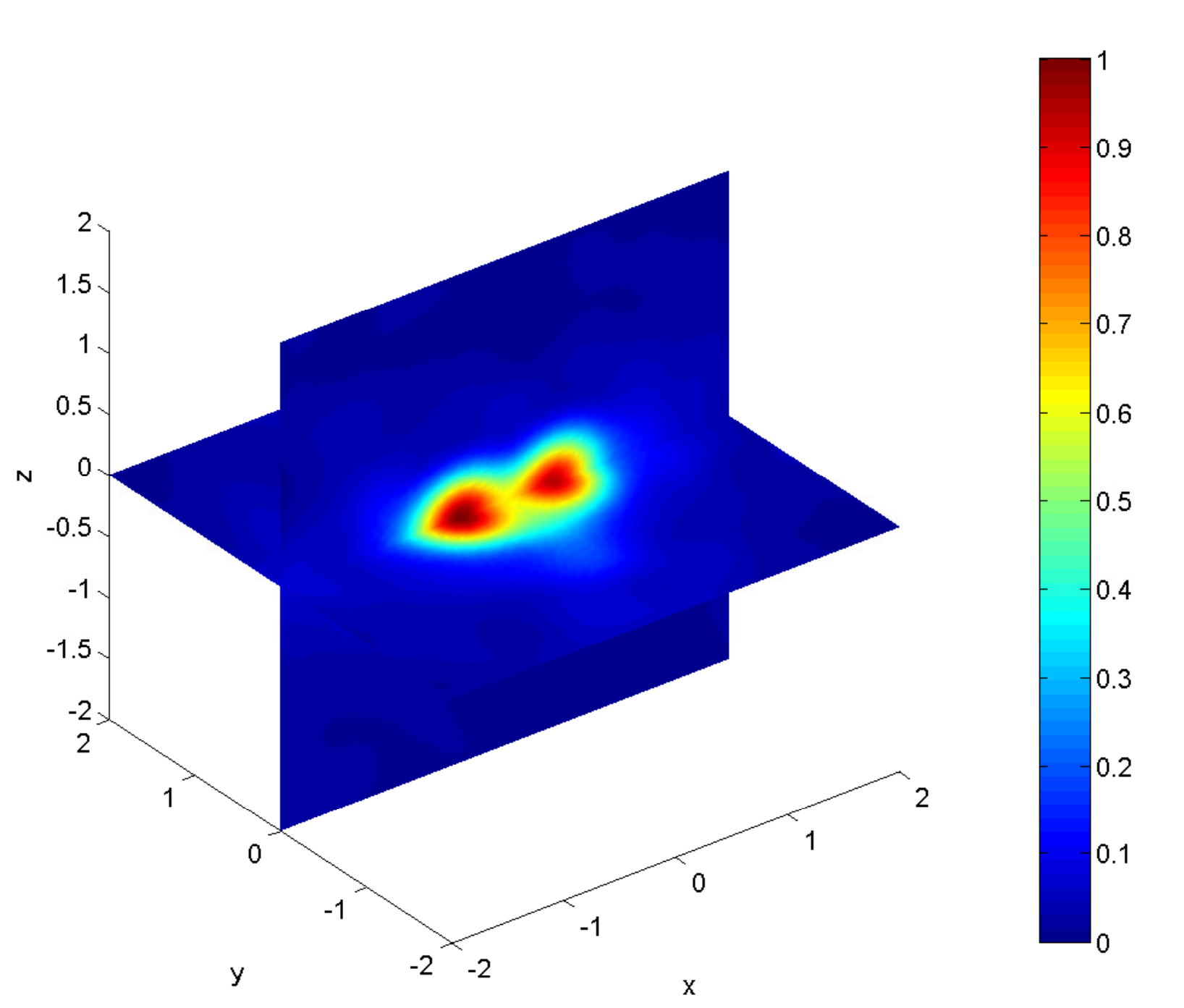}\hfill{}

\caption{\label{fig:ex4} Example 4:  Reconstruction results for Example 4:
(Left) Exact far-field data; (Right) Noisy far-field data with
$\delta=20\%$.}
\end{figure}

%

\smallskip{}



\smallskip

\subsection{SSM(r) for regular-sized scatterers}

In this subsection, we shall consider two examples to demonstrate the capability and effectiveness of the {\bf SSM(r)} for locating regular-sized scatterers.

\medskip{}

\noindent {\bf Example 5.}~In this example, we consider a revolving peanut and a revolving kite,
both revolving solids along the $x$-axis. The gravitational center of
the peanut is chosen at $(0,\,0,\,1.5\lambda)$, and the center of
the kite is anchored at $(0,\,0,\,-1.5\lambda)$. Both scatterers are taken to be PEC obstacles.

The admissible reference class is chosen to be composed of a unit ball, a revolving kite and a revolving peanut.  All
those admissible reference obstacles are centered at the origin and are PEC obstacles. Following the {\bf SSM(r)} algorithm, their norms of far-field data
associated with the reference obstacles in the admissible space are sorted in the descending order with the first one being the kite, the second one being the peanut and the third one being the unit ball. Using those a priori admissible known far-field data, we implement the
\textbf{SSM(r)}  method. The orthogonal contour slices with a certain transparency are shown in Fig.~\ref{fig:ex5} for better visualization. For both the noise-free and noisy far-field data, the {\bf SSM(r)} can successfully determine the location of the kite through
the first indicator function; see the dark red part in the center of the kite in the left figures. Once the kite is determined and its surrounding subregion trimmed from the sampling domain, we find that the second indicator function plot shows that the center of the peanut can be identified by continuing the {\bf SSM(r)}, see the right figures. After the peanut is determined and trimmed from the sampling domain, if one continues the {\bf SSM(r)} by testing data of the reference ball obstacle, no
significant peak value would be found for the reference unit ball obstacle. Our numerical results are nicely consistent with our theoretical predication.

\begin{figure}
\hfill{}\includegraphics[width=0.4\textwidth]{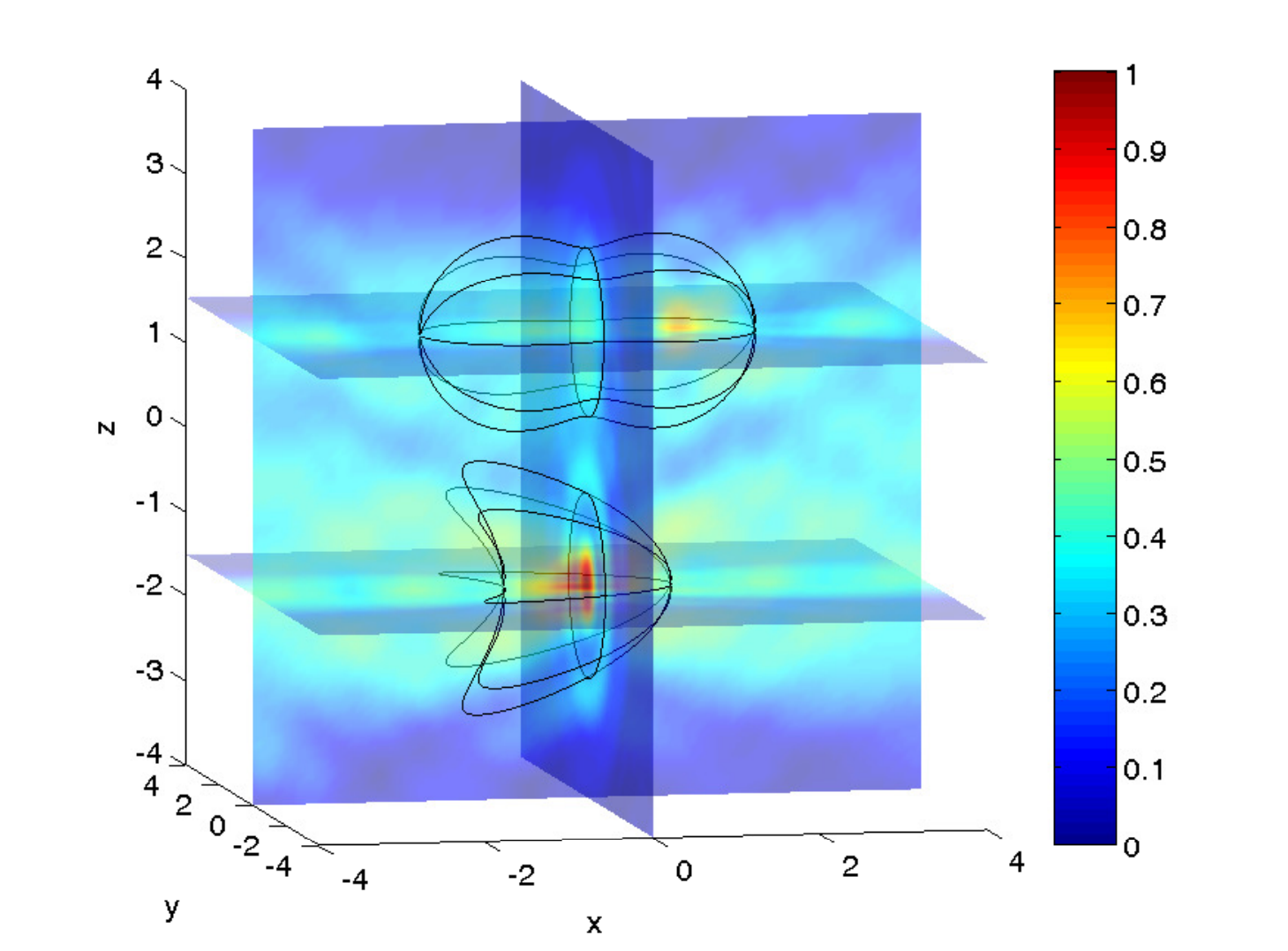}
\hfill{}\includegraphics[width=0.4\textwidth]{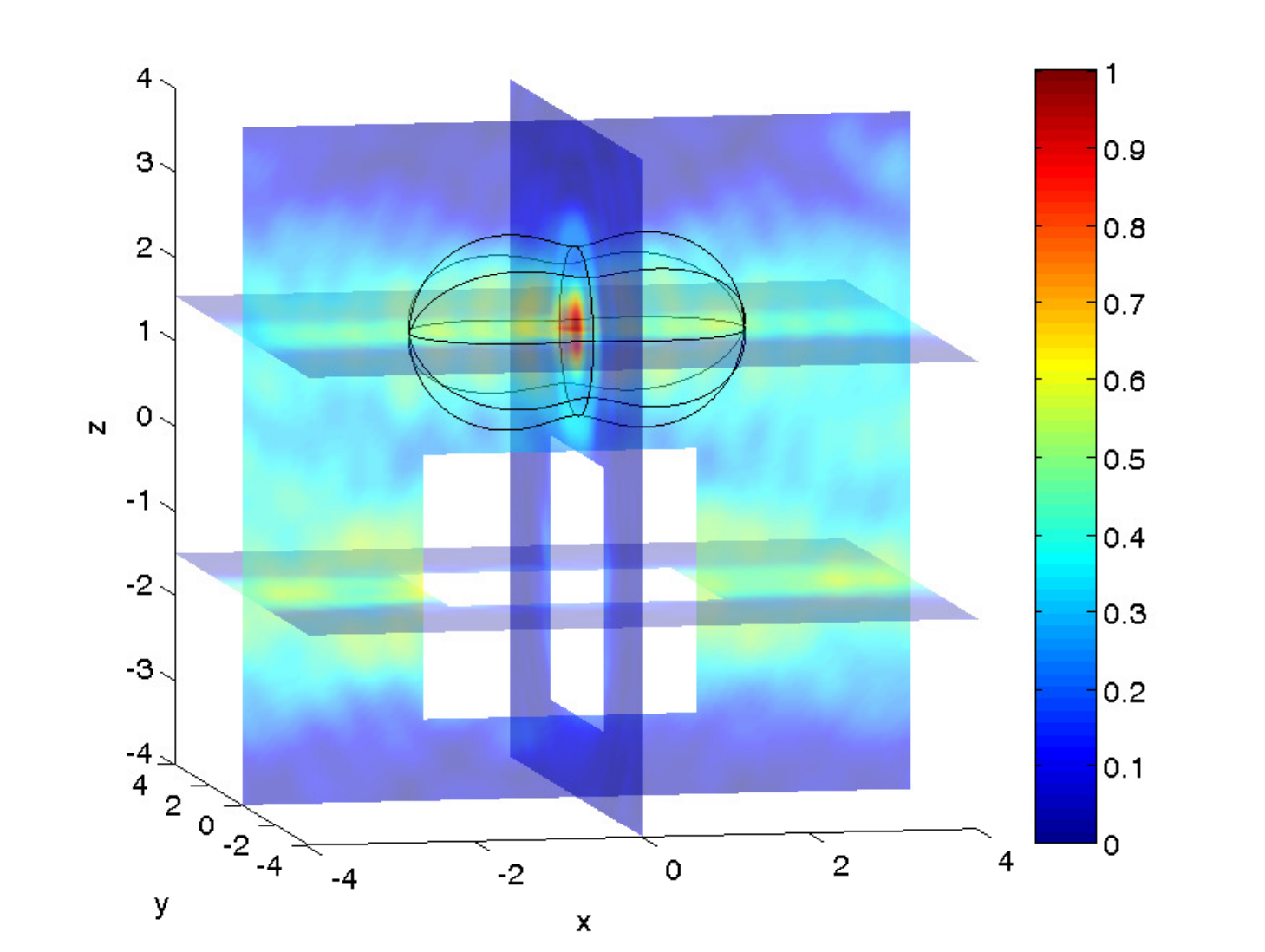}\hfill{}

\hfill{}\includegraphics[width=0.4\textwidth]{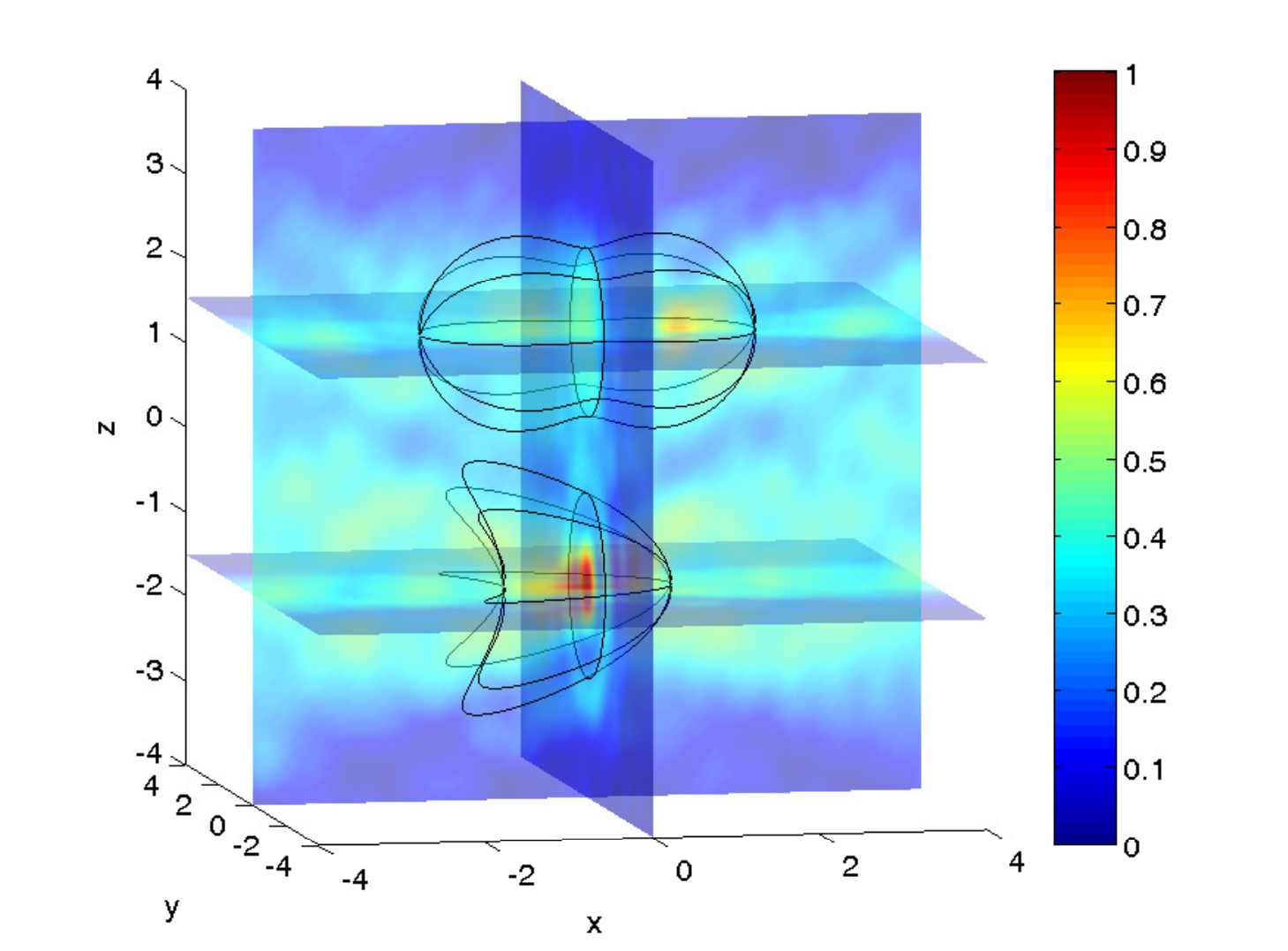}
\hfill{}\includegraphics[width=0.4\textwidth]{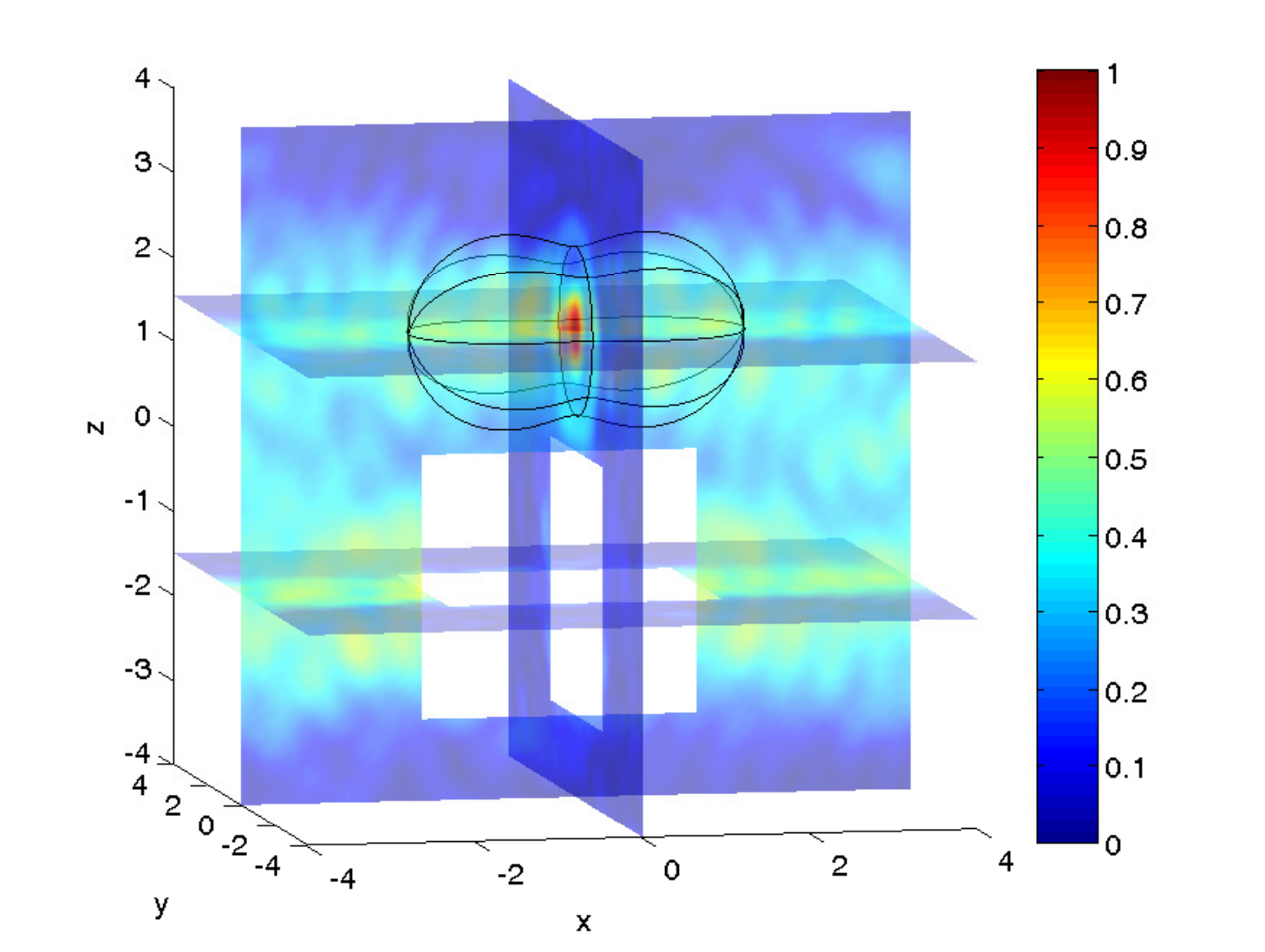}\hfill{}

\caption{\label{fig:ex5} Example 5: Reconstruction results using
(top) exact far-field data, (bottom) noisy far-field data with
$\delta=5\%$.}
\end{figure}

\medskip

\noindent\textbf{Example 6}. ~In this example, we adopt the same setting as that in Example 5, but with all the scatterer components set to be inhomogeneous media, with the  EM parameters $\varepsilon=4$, $\mu=1$ and $\sigma=0$. It can be verified that the far-field data for the three reference scatterers are distinct from each other, hence the generic condition \eqref{eq:generic assumption} is satisfied and thus the {\bf SSM(r)} applies; see Remark~\ref{rem:extension S}.
Different from the obstacle case, the order of the norms of the far-field data of medium components in the admissible reference space in Example 6 is as follows, the peanut comes first, and then the kite and finally is the unit ball. The numerical results are given in Fig.~\ref{fig:ex6}, from which we can see that the {\bf SSM(r)} can also successfully determine the locations of each medium  component successively.

%
%

\begin{figure}
\hfill{}\includegraphics[width=0.4\textwidth]{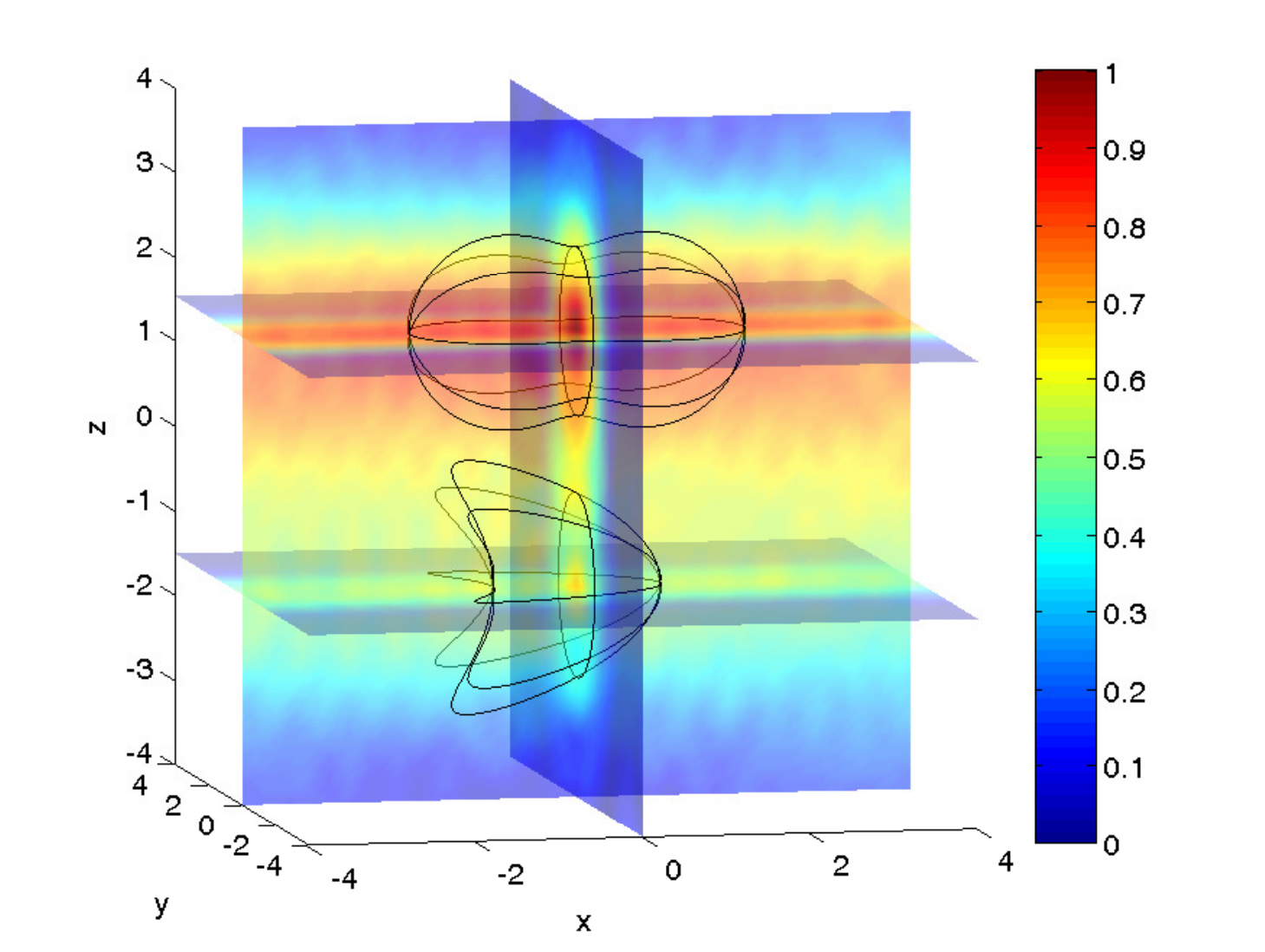}
\hfill{}\includegraphics[width=0.4\textwidth]{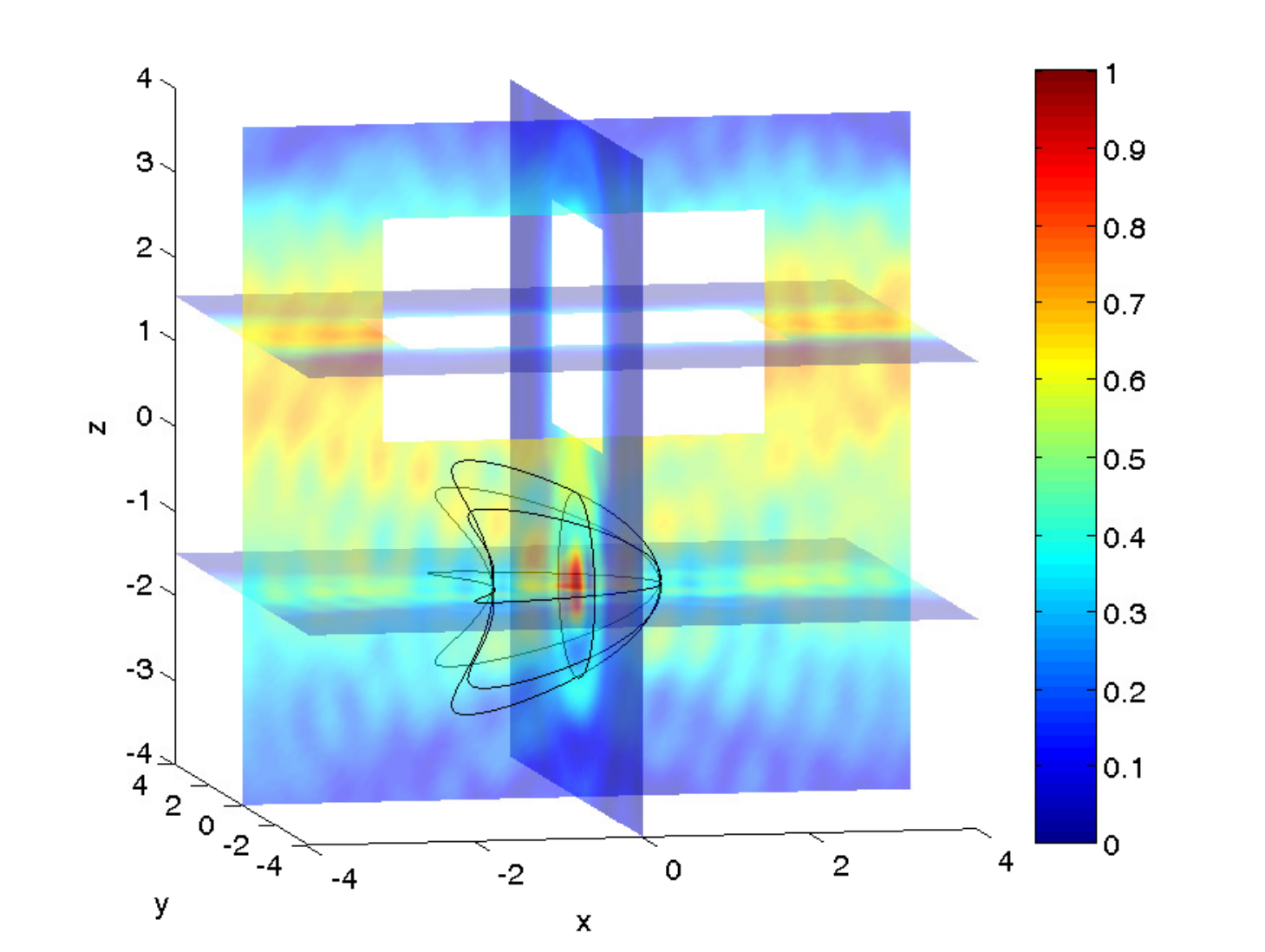}\hfill{}

\hfill{}\includegraphics[width=0.4\textwidth]{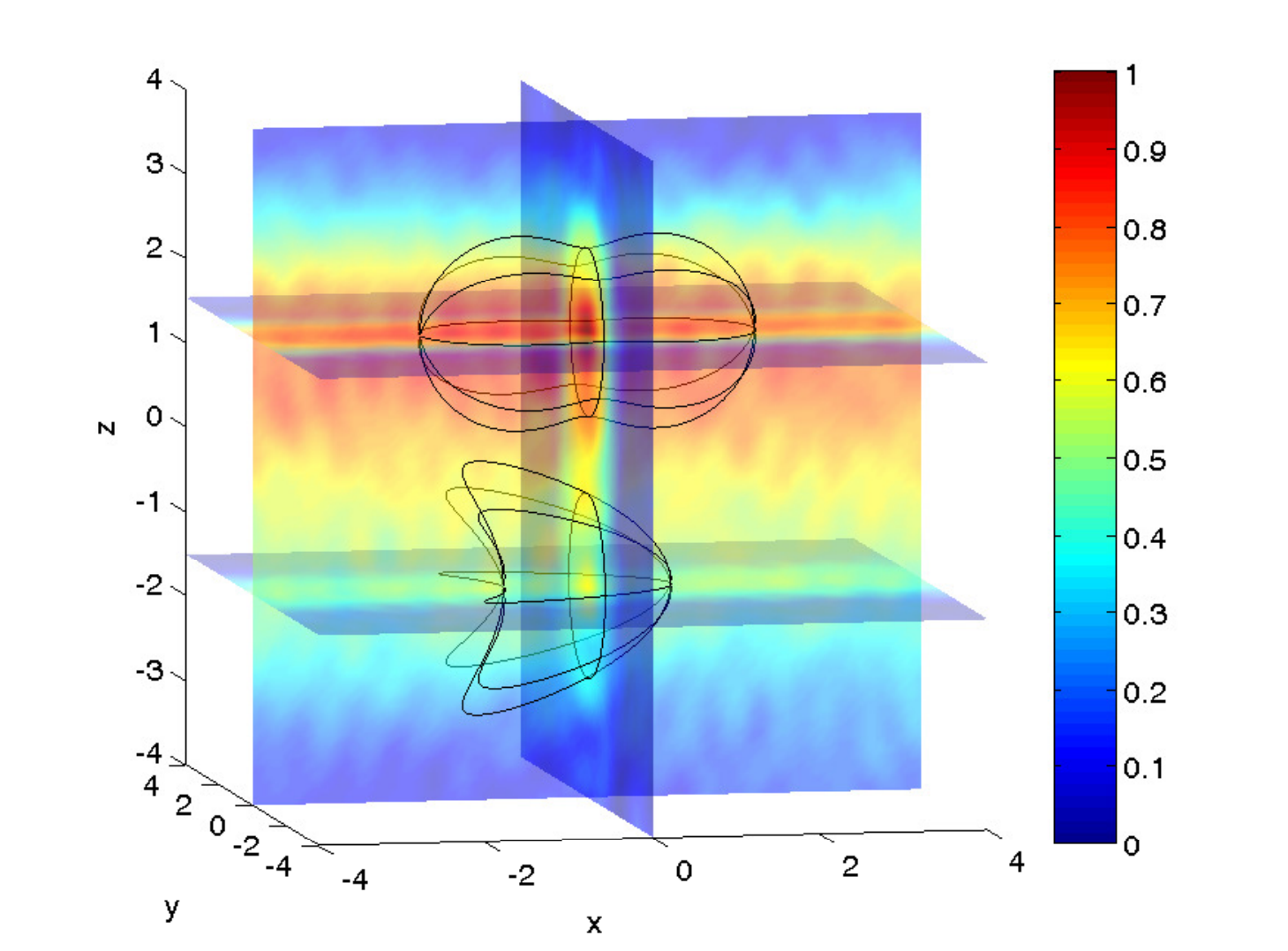}
\hfill{}\includegraphics[width=0.4\textwidth]{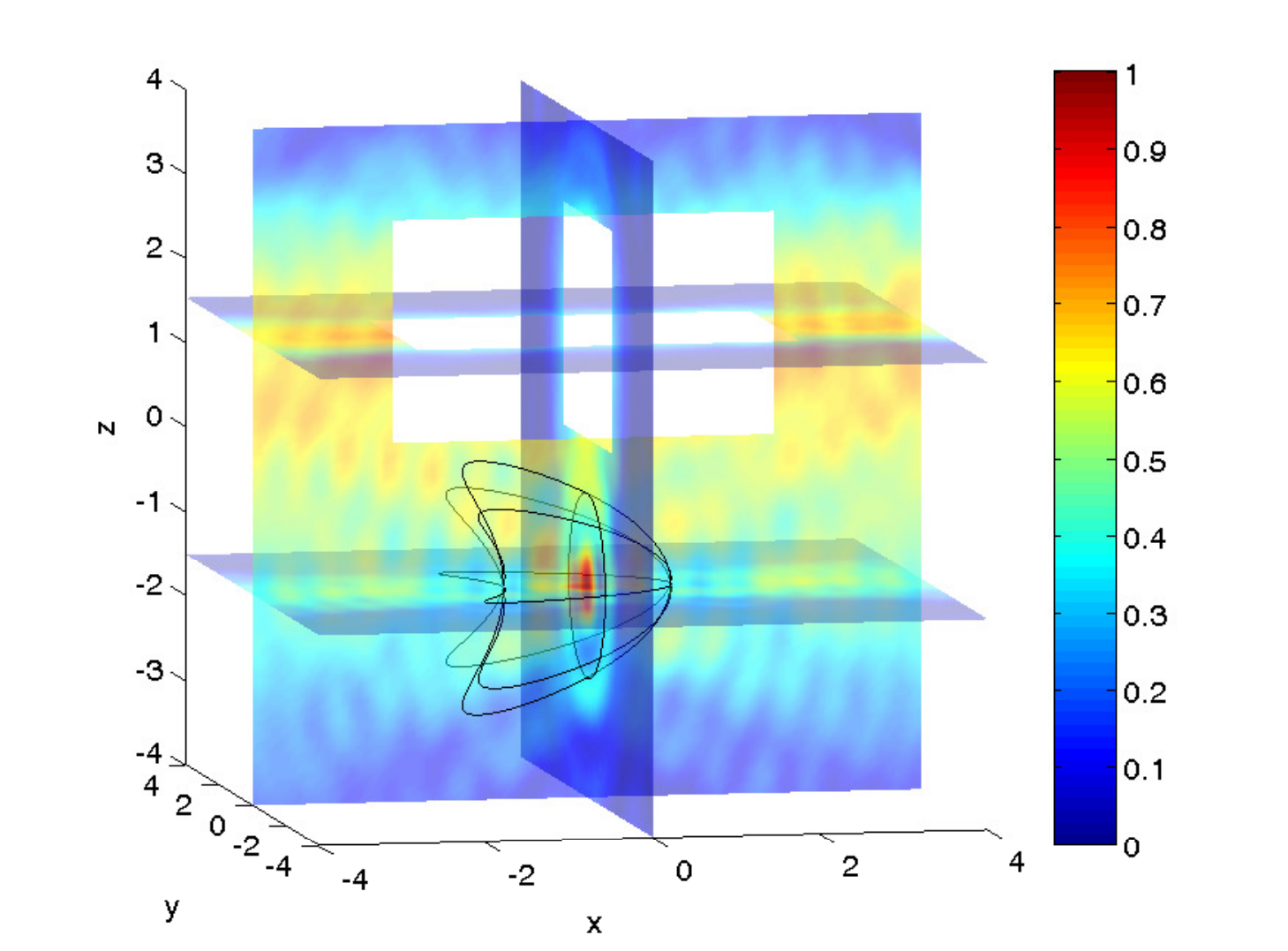}\hfill{}

\caption{\label{fig:ex6} Example 6: Reconstruction results using
(top) exact far-field data, (bottom) noisy far-field data with
$\delta=5\%$.}
\end{figure}

\smallskip{}

%
%

\smallskip{}

\subsection{SSM(s) for line-segment-like scatterers}

In this section, we present some numerical examples to show some very interesting and promising features of the {\bf SSM(s) } that were not covered in our theoretical analysis. Specifically, we shall test the performance of the {\bf SSM(s)} in identifying `partially small' line-segment-like scatterers.

\medskip{}

\noindent\textbf{Example 7.}~This
example considers a slender cylinder scatterer with base point
anchored at $(-0.5\lambda,\,0,\,0)$, radius $0.1\lambda$, height
$2\lambda$ and pointing to the positive $x$-axis. The EM parameters inside the slender cylinder are chosen to be $\varepsilon=4, \mu=1$ and $\sigma=0$. The numerical results are given in Fig.~\ref{fig:ex7}. In the noise-free case, the \textbf{SSM(s)} can 
determine successfully the location and even the length for the slender cylinder
scatterer. The identified geometry is roughly a red long bar
with the correct length. Furthermore, the
\textbf{SSM(s)} also performs well for the even more challenging case with large random noise, up to $\delta=20\%$, attached to the far-field data.

\begin{figure}
\hfill{}\includegraphics[width=0.4\textwidth]{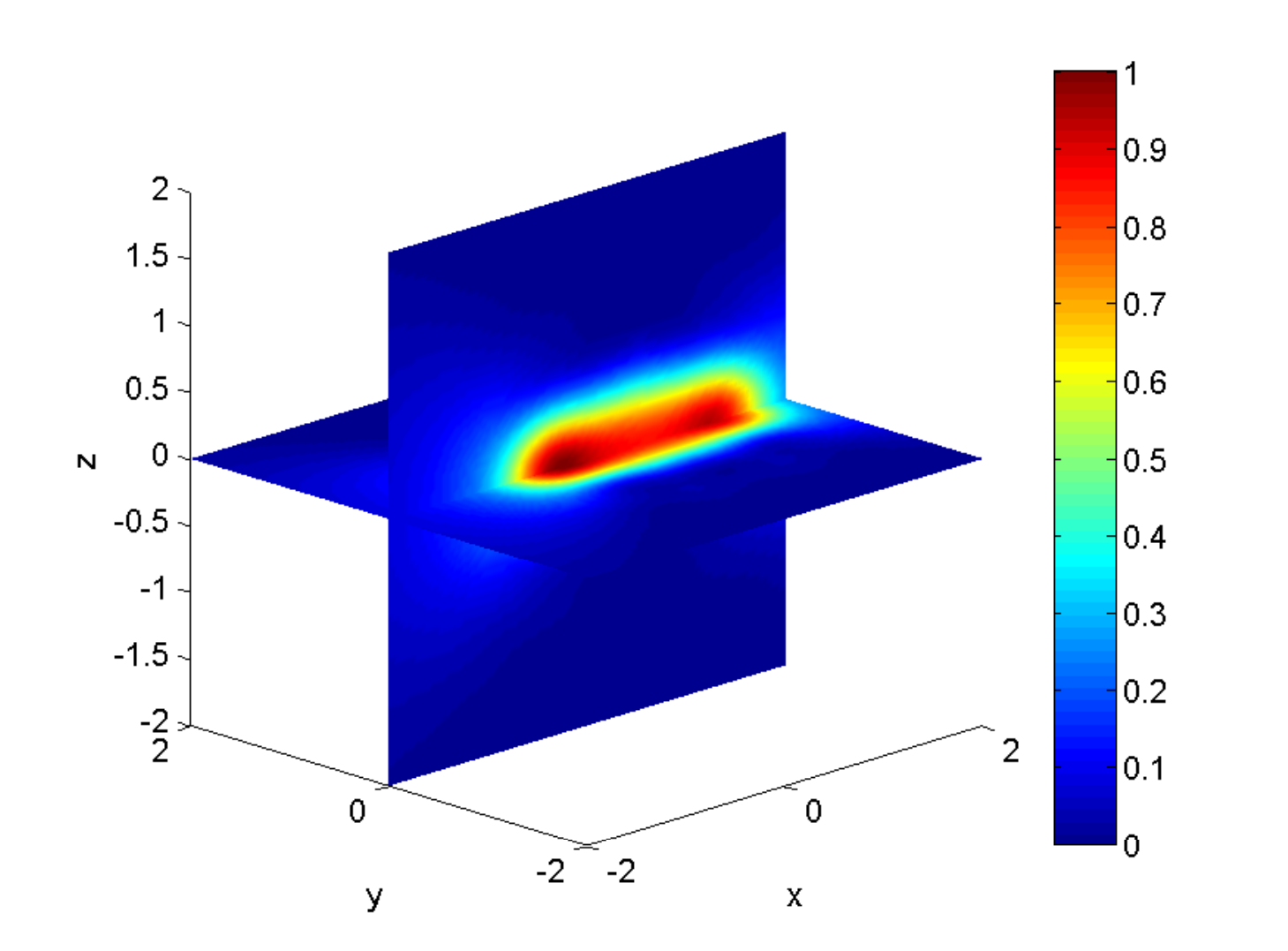}
\hfill{}\includegraphics[width=0.4\textwidth]{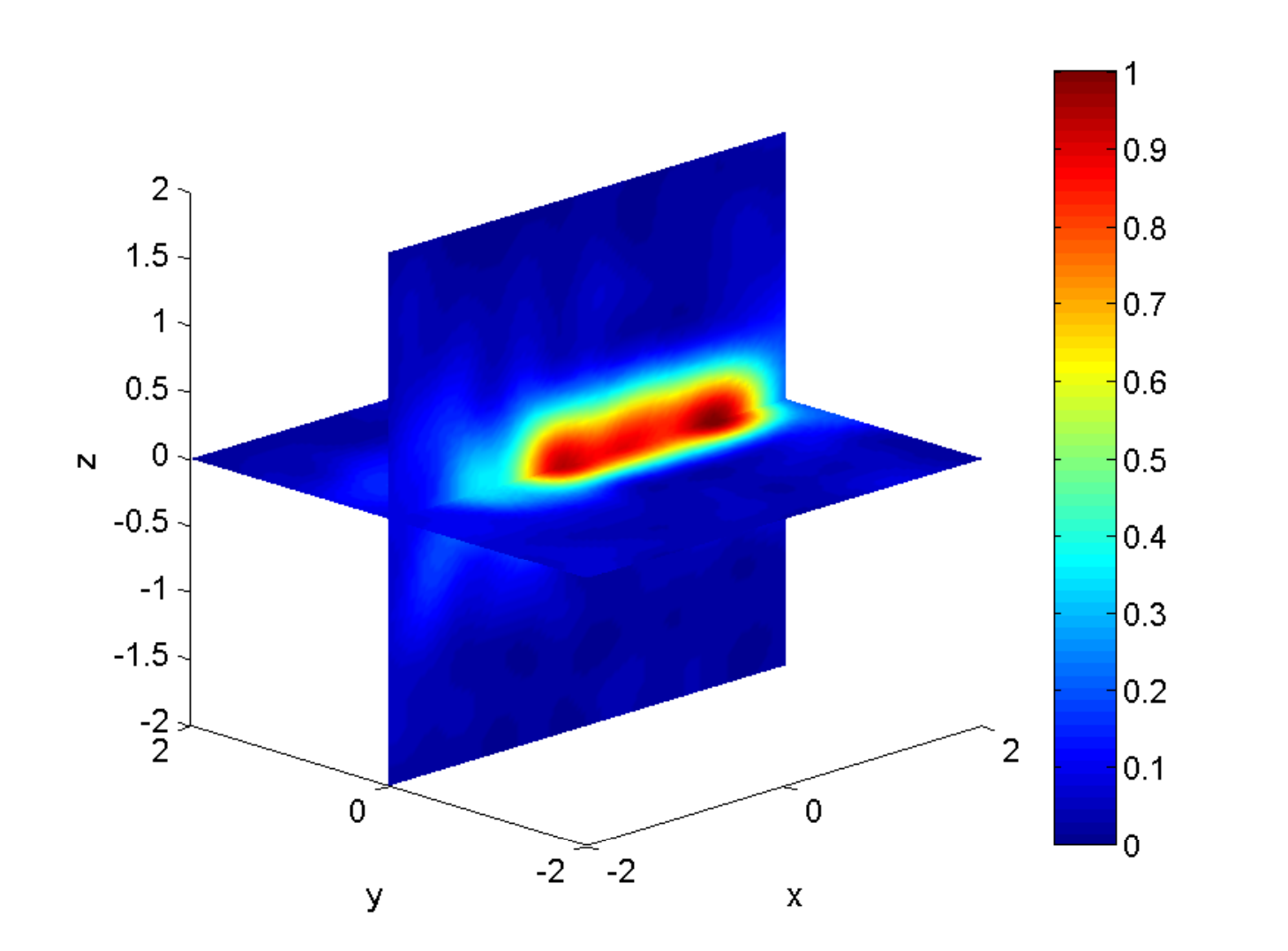}\hfill{}

\caption{\label{fig:ex7} Example 7: Reconstruction results using
(left) exact far-field data, (right) noisy far-field data with
$\delta=20\%$.}
\end{figure}

\medskip

\noindent\textbf{Example 8.} In this example, we consider an L-shaped scatterer composed of two slender cylinders as
given in Example 7, except that both base points are moved to
$(-\lambda,\,-\lambda,\,0)$ and pointing to the positive $x$- and
$y$-axes, respectively. The EM parameters inside the scatterer are chosen to be $\varepsilon=4, \mu=1$ and $\sigma=0$. For this example, it turns out that we need multiple probing wave measurements in order to have a fine reconstruction. Specifically, we shall take three far-field measurements corresponding to  $(p_1,\theta_1')=(e_3, e_1)$, $(p_2,\theta_2')=(e_1, e_2)$ and $(p_3,\theta_3')=(e_2, e_3)$. The indicator function that we shall use for the identification is given by taking the maximum value of the three separate indicator functions corresponding to the three far-field measurements. That is
\[
I_s(z)={\max}_{1\leq l\leq 3}\ I_s^{(l)}(z),\quad z\in\mathcal{T},
\]
where $I^{(l)}_s(z)$ is the indicator function calculated by using the far-field data generated by the plane wave with $(p_l, \theta_l')$, $l=1,2,3$.
It can be seen from Fig.~\ref{fig:ex8} that an L-shaped dark red bar is identified from the composite indicator function. The reconstruction result of the {\bf SSM(s)} is even robust to noise up to $20\%$ as well.

\begin{figure}
\hfill{}\includegraphics[width=0.4\textwidth]{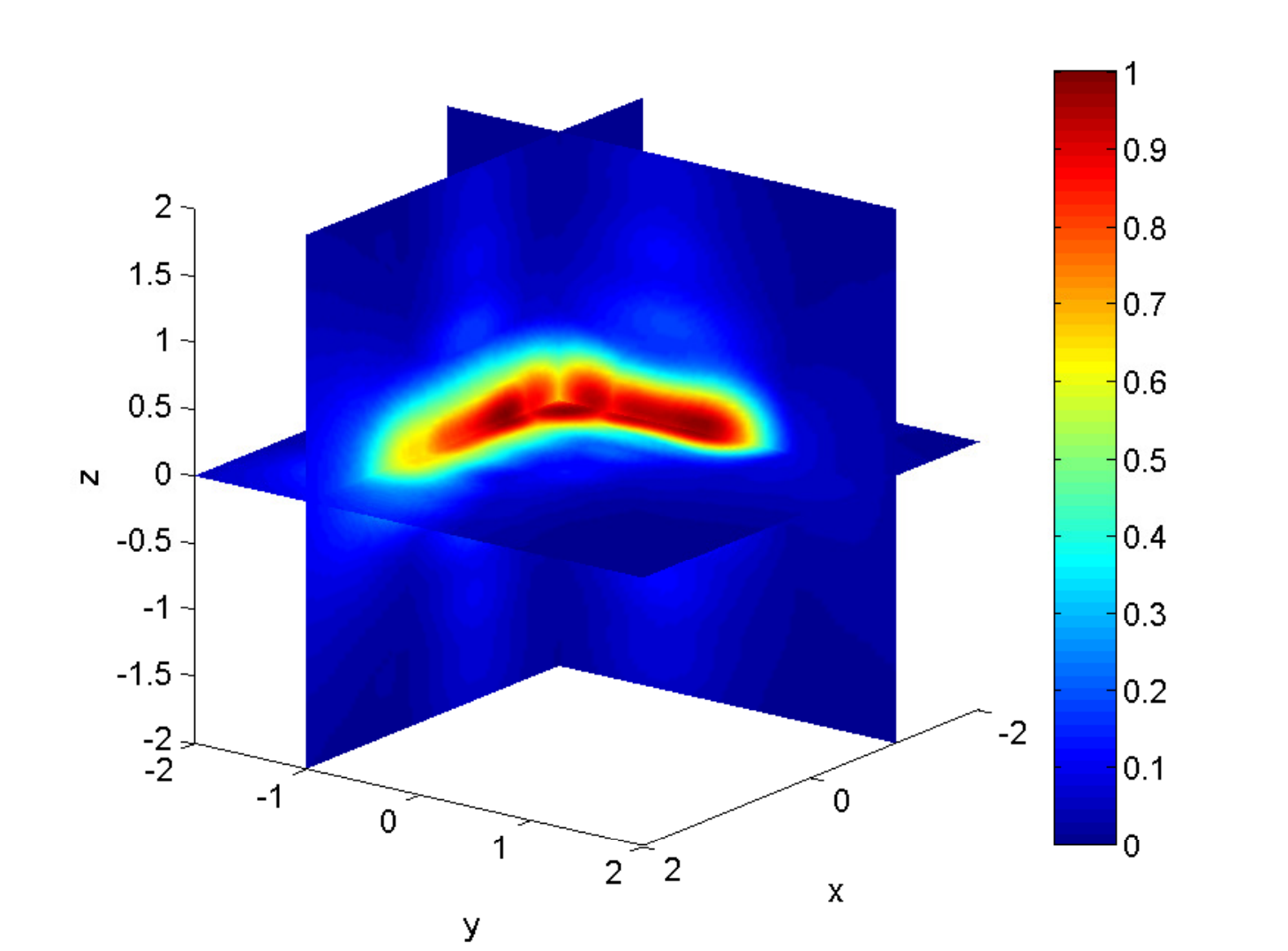}
\hfill{}\includegraphics[width=0.4\textwidth]{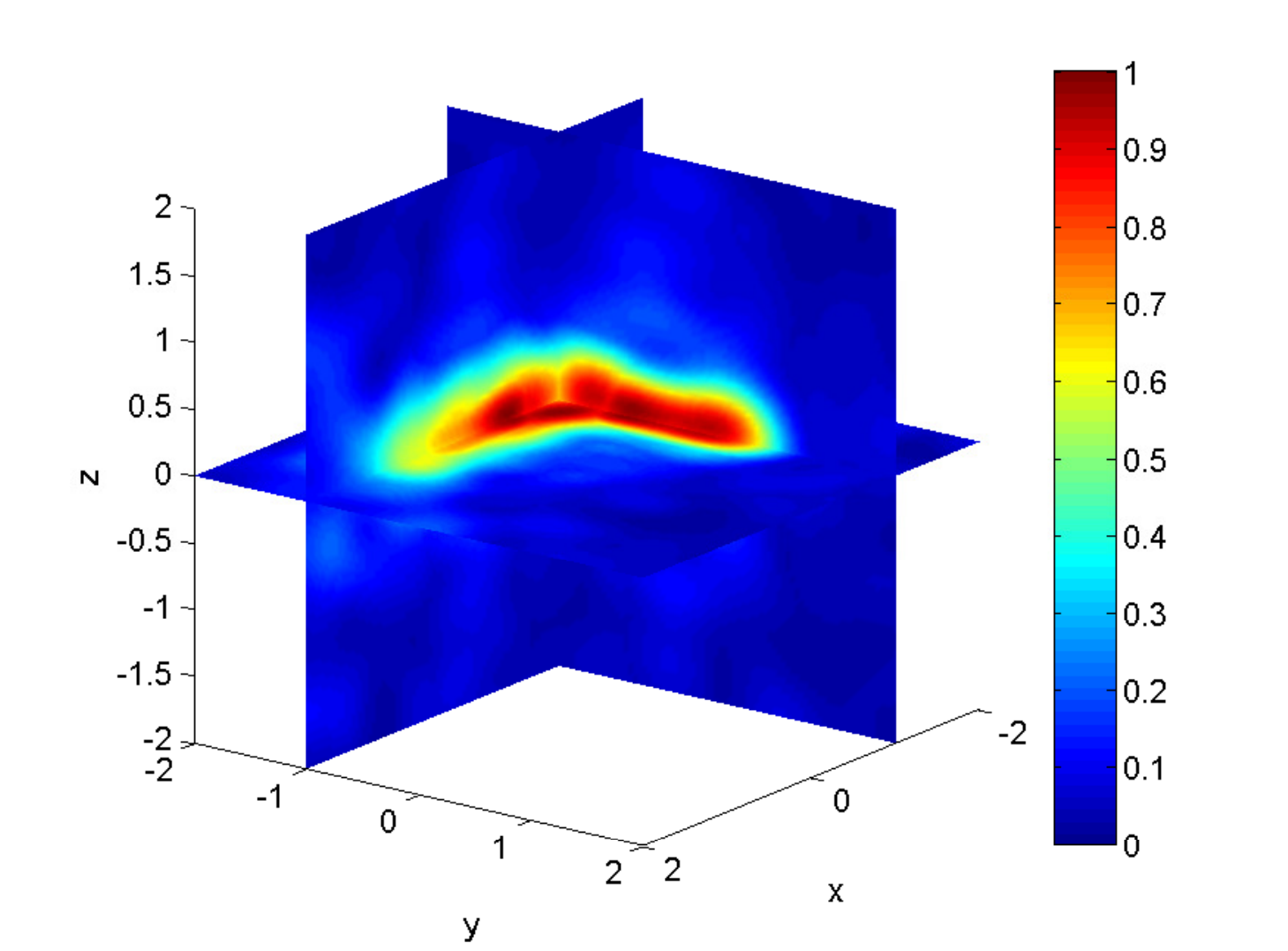}\hfill{}

\caption{\label{fig:ex8} Example 8:  Reconstruction results with
three excitation fields using (left) exact far-field data, (right)
noisy far-field data with $\delta=20\%$.}
\end{figure}

\section{Concluding remarks\label{sec:Conclusion}}

In this paper, two inverse scattering schemes, {\bf SSM(s)} and {\bf SSM(r)}, are proposed for locating multiple electromagnetic scatterers by a single electric far-field measurement. The locating schemes could work in an extremely general setting: the underlying scatterer might include, at the same time, obstacle components and inhomogeneous medium components; the number of the scatterer components and the physical property of each component are not required to be known in advance. The first scheme {\bf SSM(s)} is for locating scatterers of small size compared to the detecting EM wavelength. For this scheme, the content of each medium component is not required to known in advance either. The second scheme {\bf SSM(r)} is for locating scatterers of regular size compared to the detecting EM wavelength. For this scheme, certain {\it a priori} information would be required of each scatterer component. Specifically, if the component is an obstacle, then its shape must be from a certain known admissible reference scatterer class; and if the component is an inhomogeneous medium, then its support and content must also be from a certain known admissible reference scatterer class, and moreover, a certain generic condition must be satisfied. Nevertheless, the scheme {\bf SSM(r)} could also work in a very general setting. The reference class may consist of multiple different reference scatterers, and some reference scatterer may not be presented as a component in the unknown scatterer, and some may be presented as components for more than one time. The setting considered would be of significant interests, e.g., in radar and sonar imaging. The locating schemes are based on some novel indicator functions, whose indicating behaviors could be used for identification in many applications. In calculating the indicator functions, no inversions would be involved, so the proposed methods are very efficient and robust to noise.
Rigorous mathematical justifications are provided for both schemes. Extensive numerical experiments are conducted to illustrate the effectiveness and robustness of the proposed imaging schemes in various practical scenarios. The numerical results match our theoretical predications in a very sound manner. Furthermore, the numerical results also reveal some very interesting and promising features of the proposed schemes. The scheme {\bf SSM(s)} is also capable of qualitatively imaging the supports/shapes of the unknown scatterers in addition to locating them. Second, it is also capable of qualitatively identifying `partially small' line-segment-like scatterers. Both of these topics are worth further investigation.

\section*{Acknowledgement}

The work of Jingzhi Li is supported by the NSF of China (No. 11201453 and 91130022). The work of Hongyu Liu is supported by NSF grant, DMS 1207784. The work of Zaijiu Shang is supported by NSF of China (No.10990012). The work of Hongpeng Sun is supported the SFB Research Center "Mathematical Optimization and Application in Biomedical Sciences".

\end{document}